\documentclass{mod}
\usepackage{url}
\usepackage{setspace}
\usepackage{scrextend}
\usepackage{cite}
\usepackage{multirow}
\usepackage{cancel}
\usepackage{amsmath}
\usepackage{amsthm}
\usepackage{amssymb}
\usepackage{yhmath}
\usepackage{mathtools}
\usepackage{mathrsfs}
\usepackage[all]{xy}
\usepackage[toc,page]{appendix}
\usepackage{titlesec}
\usepackage[nottoc]{tocbibind}
\usepackage[titles]{tocloft}

\usepackage{indentfirst}

\usepackage{ stmaryrd }

\usepackage{subcaption}
\usepackage[shortlabels]{enumitem}

 \newcounter{mainthm}

\newtheoremstyle{thm_style}% 〈name〉
{10pt}% 〈Space above〉1
{7pt}% 〈Space below 〉1
{\itshape}% 〈Body font 〉
{}% 〈Indent amount 〉2
{\bfseries}% 〈Theorem head font 〉
{}% 〈Punctuation after theorem head 〉
{.5em}% 〈Space after theorem head 〉3
{}% 〈Theorem head spec (can be left empty, meaning ‘normal’ )〉

\theoremstyle{thm_style}

\newtheorem{thm}{Theorem}[section]
\newtheorem{lem}[thm]{Lemma}
\newtheorem{prop}[thm]{Proposition}

\newtheorem{cor}[thm]{Corollary}
\newtheorem{defn-thm}[thm]{Definition-Theorem}

\newtheorem{defn-lem}[thm]{Definition-Lemma}

\newtheorem{defn}[thm]{Definition}

\newtheoremstyle{rmk}% 〈name〉
{5pt}% 〈Space above〉1
{5pt}% 〈Space below 〉1
{}% 〈Body font 〉
{}% 〈Indent amount 〉2
{\bfseries\itshape}% 〈Theorem head font 〉
{}% 〈Punctuation after theorem head 〉
{.5em}% 〈Space after theorem head 〉3
{}% 〈Theorem head spec (can be left empty, meaning ‘normal’ )〉

\theoremstyle{rmk}
\newtheorem{rmk}[thm]{Remark}

\newtheoremstyle{note}% 〈name〉
{5pt}% 〈Space above〉1
{5pt}% 〈Space below 〉1
{\itshape}% 〈Body font 〉
{10pt}% 〈Indent amount 〉2
{\bfseries}% 〈Theorem head font 〉
{}% 〈Punctuation after theorem head 〉
{.5em}% 〈Space after theorem head 〉3
{}% 〈Theorem head spec (can be left empty, meaning ‘normal’ )〉

\theoremstyle{note}

\newcommand{\mathds}[1]{\text{\usefont{U}{dsrom}{m}{n}#1}}
\newcommand{\one}{\mathds {1}}
\newcommand{\m}{\mathfrak m}

\newcommand{\ml}{\mathfrak l}

\newcommand{\mk}{\mathfrak k}

\newcommand{\q}{\mathfrak q}

\titleformat{\subsubsection}[runin]{\itshape\normalsize}{\S \thesubsubsection\ }{0em}{}[\mbox{. } ]
\titlespacing{\subsubsection}
{0pt}% left
{2.5ex plus 1ex minus .2ex}% before
{0pt}% after
\usepackage{enumitem}
\setlist[description]{font=\normalfont\itshape\textbullet\space}

\usepackage{titletoc}

\titlecontents{section}
[0.72em]
{\scriptsize
	\bfseries
	\vspace{2pt}}%
{\thecontentslabel.\enspace}%numbered sections
{}%numberless section
{\titlerule*[0.38pc]{.}\contentspage}%

\titlecontents{subsection}
[0em]
{\scriptsize \vspace{0pt}}%
{\qquad \quad \thecontentslabel.\enspace}%numbered sections
{}%numberless section
{\titlerule*[0.5pc]{.}\contentspage}%

\titlecontents{subsubsection}
[0em]
{\footnotesize \vspace{1pt}}%
{\qquad \qquad \thecontentslabel.\enspace}%numbered sections
{}%numberless section
{\titlerule*[0.5pc]{.}\contentspage}%

\titleformat{\subsection}[runin]{
\bfseries
\itshape
	\normalsize}{\thesubsection) }{0em}{}[\mbox{ . } ]
\titlespacing{\subsection}
{0pt}% left
{2.25ex plus 1ex minus .2ex}% before
{0pt}% after

\numberwithin{equation}{section}
\renewcommand{\theequation}{\thesection.\arabic{equation}}

\newcommand{\BetweenLarge}{\fontsize{13.5pt}{16.2pt}\selectfont}

\allowdisplaybreaks

\begin{document}
\title
[
%Batalin-Vilkovisky structure on open-closed Hochschild cohomology
Cyclic brace relation and Batalin-Vilkovisky structure on open-closed Hochschild cohomology
]{ 
\BetweenLarge
Cyclic brace relation and BV structure on open-closed Hochschild cohomology 
}

\author[Hang Yuan]{ Hang Yuan }

\begin{abstract} {\sc Abstract:}
For an open-closed homotopy algebra (OCHA), the previous work indicates that there is an open-closed version of Hochschild cohomology with a canonical Gerstenhaber algebra structure.
If this OCHA is further cyclic and unital in the sense of Kajiura and Stasheff, we produce a BV algebra structure on this cohomology via a cochain-level identity formulated with cyclic brace operations.
\end{abstract}

\maketitle
%
% Title
%

%\tableofcontents

\hypersetup{
	colorlinks=true,
	linktoc=all,
	%	linkcolor=gray,
	citecolor=gray
}

%----set length
\setlength{\parindent}{5.5mm}	\setlength{\parskip}{0em}

\section{Introduction}

Let $X$ be a Calabi-Yau manifold. The space of polyvector fields on 
$X$ carries a Batalin-Vilkovisky (BV) algebra structure \cite{schechtman1998remarks, ran1997thickening}, analogous to how the space of polyvector fields on a smooth manifold forms a Gerstenhaber algebra.
The Hochschild cochain complex 
\begin{equation}
	\label{C_AA_intro_eq}
C^\bullet(A,A)=\prod_{k\geqslant 1} \mathrm{Hom} (A^{\otimes k}, A)
\end{equation}
of an associative or $A_\infty$ algebra 
$A$ serves as the non-commutative geometric analogue of polyvector fields. Its cohomology $HH(A,A)$ is called \textit{Hochschild cohomology} and naturally inherits a Gerstenhaber algebra structure \cite{Gerstenhaber_1963cohomology}.
A BV algebra can be viewed as an algebra over the homology of the framed little disk operad \cite{getzler1994batalin}; explicitly, it is a Gerstenhaber algebra $(G,\smallsmile, [,])$ further equipped with an operator $\Delta$ such that $\Delta^2=0$ and for any $a,b\in G$,
\begin{equation}
	\label{BV_eq_intro}
	[a, b] = \Delta (a \smallsmile b)  + \Delta a \smallsmile b - (-1)^{|a||b|} \   \Delta b \smallsmile a
\end{equation}

A BV algebra structure exists on the Hochschild cohomology of certain special classes of associative or $A_\infty$ algebras.
Tradler first established such a structure for unital associative or $A_\infty$ algebras with a symmetric inner product \cite{tradler2008batalin} (see also \cite{menichi2009batalin}). Ginzburg proved that the Hochschild cohomology of a Calabi-Yau algebra also admits a BV structure \cite{ginzburg2006calabi}.
Lambre introduced the notion of a differential calculus with duality that explains when BV structure exists and unifies the two cases of symmetric algebras and Calabi-Yau algebras \cite{lambre2009dualit}.
Besides, BV algebra structures persist for twisted Calabi-Yau algebras and Frobenius algebras with semisimple Nakayama automorphisms, studied by Kowalzig-Kr\"ahmer \cite{kowalzig2014batalin} and Lambre-Zhou-Zimmermann \cite{lambre2016hochschild} respectively.

The Hochschild cochain complex of an associative or $A_\infty$ algebra $(A,\m)$ carries a family of operations known as \textit{braces} or brace operations
$
D, E_1,\dots, E_n \rightsquigarrow  D\{E_1,\dots, E_n \}
$ (see \cite{Kadeishvili_1988structure, Getzler_1993cartan}).
The cup product $\smallsmile$ and Gerstenhaber bracket $[  ,  ]$ on Hochschild cohomology $HH(A,A)$  admit cochain-level descriptions via brace operations:
$D\smallsmile E = \m\{D, E\} $ and $ [D,E]=D\{E\} \pm E\{D\}$.
In fact, the Gerstenhaber algebra structure, such as the compatibility condition between $\smallsmile$ and $[,]$, can be entirely derived from the so-called \textit{brace relation} at the cochain level; see e.g. \cite{Yuan_ocha_hochschild}.
Now, if $(A,\m)$ is cyclic or Calabi-Yau, while we know $HH(A,A)$ carries a BV algebra structure, we should be also able to find identities on Hochschild \textit{cochains} that induces the BV relation (\ref{BV_eq_intro}) upon passing to Hochschild \textit{cohomology}.

\subsection{Main result}
Given a non-degenerate bilinear form $\omega:A\otimes A\to \Bbbk$, one can define the \textit{cyclic brace operation}
$
D\{E_1,\dots, E_s, \Delta, E_{s+1},\dots, E_n \}
$
(see Definition \ref{cyclic_brace_defn}) which jointly generalizes the BV-operator $\Delta$ (see (\ref{Delta_eq})) and the usual brace operation.
Indeed, $D\{\Delta\}= \Delta D$. This essentially comes from the notion of symbols in \cite[Definition 14]{tradler2008batalin} and 
the notion of spined braces in \cite[Section 4.1]{ward2012cyclic}.
We note that the cyclic braces admit a natural extension to the open-closed Hochschild-type space (\ref{C_Hochschild_ZA_intro_eq}). This is then used to prove our main result below.

For any Hochschild cochains $D, E$, the cyclic braces allow us to achieve a cochain-level identity underlying the BV relation (Theorem \ref{[DE]_thm}):
\begin{equation}
	\label{[DE]_intro}
	\begin{aligned}
		[D, E] &= \Delta (\m\{D, E\} ) + \m \{ \Delta D, E\} - (-1)^{|D||E|} \m\{\Delta E, D\}  \\
		&\qquad\qquad\qquad\qquad +
		\delta \left( D\{E,\Delta\} \right) - (-1)^{|D||E|}  \delta\left( E\{D, \Delta\}  \right) \\
		&\qquad\qquad\qquad\qquad +
		(\delta D) \{E, \Delta\} - (-1)^{|D||E|} (\delta E) \{D, \Delta\} \\
		&\qquad\qquad\qquad\qquad +
		(-1)^{|D|} \left( D\{\delta E, \Delta\} - (-1)^{|D||E|} E\{\delta D, \Delta\} \right)
	\end{aligned}
\end{equation}
where $\delta$ is the Hochschild differential.
If $\delta D =\delta E=0$, then passing to the Hochschild cohomology recovers the desired BV relation (\ref{BV_eq_intro}). 

The cochain-level identity (\ref{[DE]_intro}) does not appear explicitly in the literature, but its underlying ideas are likely implicit in existing work. For the associative algebra case restricted to $\delta$-closed elements, see \cite[Lemma 5]{tradler2008batalin}.
Nevertheless, the use of cyclic braces provides a more convenient framework, and the explicit formula above is useful for extending the standard BV structure \cite{tradler2008batalin,menichi2009batalin} to open-closed homotopy algebras (OCHA).

\begin{thm}
	\label{main_thm_intro}
	The normalized open-closed Hochschild cohomology $\overline{HH}(Z; A, A)$ of a unital and cyclic OCHA admits a canonical BV algebra structure. 
\end{thm}

The notion of \textit{{open-closed homotopy algebra}} ({OCHA}) is introduced by Kajiura and Stasheff \cite{ocha_kajiura_cmp,ocha_kajiura_survey}, drawing inspiration from Zwiebach's open-closed string field theory \cite{zwiebach1998oriented}.
Given an arbitrary OCHA, we have developed an analogue of Hochschild cohomology in \cite{Yuan_ocha_hochschild} and established its Gerstenhaber algebra structure.
Besides, following Kajiura and Stasheff, we can also introduce \textit{unital} and \textit{cyclic} OCHAs (see Definition \ref{unital_defn}).

Given a pair of spaces $(Z, A)$, we consider the Hochschild-type space
\begin{equation}
	\label{C_Hochschild_ZA_intro_eq}
	C^{\bullet,\bullet}(Z; A, A)=\prod_{\ell,k\geqslant 0, (\ell,k)\neq (0,0)} \mathrm{Hom}(Z^{\wedge \ell}\otimes A^{\otimes k} , A)
\end{equation}
and the Chevalley-Eilenberg-type space
\begin{equation}
	\label{C_Chevally-Eilenberg_Z_intro_eq}
	\tilde C^{\bullet} (Z, Z) =
	\prod_{\ell\geqslant 1} \mathrm{Hom}(Z^{\wedge \ell}, Z)
\end{equation}
where the superscript $\wedge \ell$ indicates the multilinear maps are graded symmetric for the inputs from a graded vector space $Z$.
By an \textit{open-closed homotopy algebra (OCHA)} on $(Z,A)$, we mean a pair of an $L_\infty$ algebra $\ml=\{\ml_\ell\}\in \tilde C^\bullet (Z, Z)$ and a family of multilinear maps $\q=\{\q_{\ell,k}\}$ in $ C^{\bullet,\bullet}(Z; A, A)$ with certain compatibility conditions.
Equivalently, by the coalgebra description \cite{ocha_hoefel2012coalgebra}, it is a degree one coderivation $\mathcal D$ on $\Lambda^cZ\oplus T^c A$ such that $\mathcal D^2=0$, where $\Lambda^c Z$ and $T^cA$ are the graded symmetric coalgebra and tensor coalgebra.
The OCHA provides a unified mathematical structure that combines:
(i) Open string interactions governed by 
$A_\infty$ algebras; (ii)
Closed string interactions described by $L_\infty$ algebras.
An OCHA $(\ml,\q)$ naturally gives rise to an open-closed Hochschild differential $\delta=\delta_{\ml,\q}$ on the space $C^{\bullet,\bullet}(Z; A, A)$ in the form (see (\ref{delta_ocha_eq})):
\begin{align*}
	& \delta (D) (z_1,\dots , z_\ell ; a_1,\dots a_k) \\ 
	&= \sum \pm \q(z_{J_1}; a_1,\dots, D(z_{J_2}; \dots) , \dots,  a_k) \pm D(z_{J_1}; a_1,\dots, \q(z_{J_2}; \dots) , \dots,  a_k) \pm \q(\ml(z_{J_1})\wedge z_{J_2} ; \dots)
\end{align*}
where $J_1\sqcup J_2=\{1,\dots, \ell\}$.
The cohomology $HH(Z; A, A)$ of this $\delta$-complex
is called the \textit{open-closed Hochschild cohomology}, was shown to carry a canonical Gerstenhaber algebra structure \cite{Yuan_ocha_hochschild} .
We remark that the cochain-level identity in (\ref{[DE]_intro}) still holds true if we replace $\m$ with $\q$ and take this open-closed Hochschild differential $\delta$.

Suppose an OCHA $(Z,A,\ml,\q)$ is \textit{cyclic} in the sense of Kajiura-Stasheff in \cite[Definition 15]{ocha_kajiura_cmp} and is \textit{unital} with a unit $\one\in A$ (see Definition \ref{unital_defn}).
Just as the case of associative or $A_\infty$ algebras, an open-closed Hochschild cochain $D$ in $C^{\bullet,\bullet}(Z; A, A)$ is called \textit{normalized} if it vanishes whenever $\one$ is an input (cf. \cite[1.5.7]{loday2013cyclic}).
One can show that the subspace $\overline{C}^{\bullet,\bullet}(Z; A, A)$ of normalized Hochschild cochains is preserved by the differential $\delta$.
The resulting cohomology of this subspace is denoted by $\overline{HH}(Z; A, A)$ and is called the normalized open-closed Hochschild cohomology; see Section \ref{s_normalized_Hochshild}.

\begin{rmk}
	One simple example of OCHA is the following case. Let $f:Z\to A$ be a cochain map from a cochain complex $(Z,d_Z)$ to a differential graded algebra (dga) $(A, d_A, \cdot)$; cf. \cite[Example 4.5]{Yuan_ocha_hochschild}.
	This data gives rise to an OCHA structure $(\ml,\q)$ by assigning $\ml_1=d_Z$, $\q_{0,1}=d_A$, $\q_{1,0}=f$, and $\q_{0,2}(a_1,a_2)= \pm a_1\cdot a_2$.
	Moreover, let $\varphi :N\to M$ be a smooth map between oriented manifolds, and consider the special case $Z=\Omega^*(M)$, $A=\Omega^*(N)$, and $f=\varphi^*$. Then, the data $(f, Z,A)$ induces a unital and cyclic OCHA structure. The unit is given by the constant-one function, and the pairing for the cyclic condition is $\omega(\alpha_1,\alpha_2)= \pm \int_N \alpha_1\wedge \alpha_2$.
	Via the Hochschild homology of this OCHA, we study the iterated integral model of the relative disk mapping space consisting of pairs $(\Phi,\gamma)$ of maps $\Phi\colon \mathbb D\to M$ and $\gamma\colon S^1\to N$ such that $\Phi|_{\partial\mathbb D}=f\circ\gamma$; see \cite{Yuan_Yi_relative_disk}.
	Following Irie’s de Rham chain model of string topology \cite{irie2020chain}, we further expect that one could construct an open-closed version of this model, in which the open-closed Connes operator \(B\) from \cite{Yuan_Yi_relative_disk} should be dual to the BV operator \(\Delta\) considered in this paper \eqref{Delta_eq}.
\end{rmk}

Given an OCHA $(\ml,\q)$ as above, one can verify that the subcollection $\m=\q_{0,\bullet}=\{\m_k:=\q_{0,k}\}_{k\geqslant 1}$ of $\q$ is an $A_\infty$ algebra.
Conversely, in this way, \textit{any $A_\infty$ algebra can be viewed as an OCHA}.
Moreover, if the OCHA is cyclic (resp. unital), then the above $\m=\q_{0,\bullet}$ is also a cyclic (resp. unital) $A_\infty$ algebra; cf. \cite[2.11]{kajiura2007noncommutative}, \cite[2.8]{keller2006A}.
Thus, we can retrieve:

\begin{cor}
	\label{cor_intro}
	Given a cyclic unital $A_\infty$ algebra $A$, the normalized Hochschild cohomology is a BV algebra.
\end{cor}

\begin{rmk}
We work with \textit{normalized} Hochschild cohomology solely to ensure that the BV operator $\Delta$ squares to zero. The normalized and non-normalized cochain complexes are often quasi-isomorphic. For associative algebras, Hochschild cohomology arises as the left derived functor of the Hom functor, with the normalized version remaining isomorphic via the normalized bar resolution \cite[1.5.7]{loday2013cyclic}, \cite[Exercise 3.2.3]{khalkhali2013basic}. However, for more general structures (OCHAs, dgAs, or $A_\infty$ algebras), the case is likely more subtle.
In the cyclic bar complex context, an extra connectivity condition on the $A_\infty$ algebra may be required to obtain a similar quasi-isomorphism \cite[Theorem 5.2]{getzler1990algebras}. Tradler’s result for $A_\infty$ algebras also appears to rely on a quasi-isomorphism between the normalized and non-normalized complexes but omits details \cite[Page 2373]{tradler2008batalin}.
While we expect the quasi-isomorphism to hold under reasonable conditions, to avoid introducing extra assumptions, we work directly with normalized complexes.
\end{rmk}

\begin{rmk}
A natural generalization of Theorem \ref{main_thm_intro} would relax the cyclic condition to certain Calabi-Yau condition. While Calabi-Yau algebras are well-studied (see Ginzburg’s seminal work \cite{ginzburg2006calabi}), various versions of Calabi-Yau $A_\infty$-algebras have been studied in the literature (e.g., \cite{kontsevich2008notes}, \cite[Definition 3.1.3]{ginzburg2006calabi}).
One motivation for this work comes from the non-archimedean SYZ constructions in \cite{Yuan_I_FamilyFloer}, where families of $A_\infty$-algebras associated with Lagrangian fibrations were used to construct a canonical mirror dual fibration. To further equip the mirror with a Calabi-Yau volume form, it seems reasonable to explore families of Calabi-Yau $A_\infty$-algebras (as appropriately defined). This motivates our search for a suitable definition of Calabi-Yau $A_\infty$-algebras and potentially Calabi-Yau OCHAs (a notion not yet found in the literature) that are compatible with the above ramework.
We speculate that the holomorphic volume form on the total space of a Lagrangian fibration could interact with the $A_\infty$-algebras associated with Lagrangian fibers via the evaluation maps at the interior marked points for the relevant holomorphic disks. 
Also, we propose that a suitable definition of Calabi-Yau $A_\infty$-algebras should algebraically endow the associated Hochschild cohomology with a BV algebra structure, and geometrically align with Fukaya’s $A_\infty$-algebras for Lagrangians in a Calabi-Yau manifold. 
The cyclic condition would be then a special case of an appropriate Calabi-Yau condition.
%Although a Calabi-Yau condition for an OCHA should be analogous, it does not yet appear in the literature to our knowledge. 
%Accordingly, let's temporarily focus on the open-closed Hochschild cohomology for the cyclic case and leave the Calabi-Yau generalization for future work.
\end{rmk}

\subsection{Cyclic brace relations}

A useful object of our study of Theorem \ref{main_thm_intro} is a structure we call \textit{cyclic braces}, which can be defined both in the open-string case (\ref{C_AA_intro_eq}) and in the open-closed-string case (\ref{C_Hochschild_ZA_intro_eq}). 
Let's briefly describe the former case.
Let $A$ be a graded vector space over a field $\Bbbk$, and fix an element $\one\in A$. For now, we do \textit{not} assume that $A$ carries any (unital) associative or $A_\infty$-algebra structure, nor that $\one$ serves as a unit.

Suppose there is a non-degenerate, skew-symmetric bilinear map $\omega:A\otimes A\to \Bbbk$.
Using the non-degeneracy of $\omega$, we define the following operation (Definition \ref{cyclic_brace_defn}):
For any $D, E_1,\dots, E_m\in C^\bullet (A, A)$, we construct a new element $D\{E_1,\dots, E_m, \Delta \}$ in $C^\bullet(A, A)$ by
\[
\omega \big(D\{E_1,\dots, E_m,  \Delta \} (a_1,\dots, a_k) , \ a_0 \big)
=
\sum \pm \omega \big(D(a_{r+1}, \dots, E_1 ( \ ) ,\dots, E_m (\ ),\dots, a_0,\dots, a_r ) \ , \ \one \big)
\]
In other words, up to sign, we consider the sum over all cyclic permutations of the inputs $a_1,\dots, a_k, a_0$, while the "ghost" symbol $\Delta$ serves to anchor the position of $a_0$ relative to the 
$E_i$'s.
Moreover, we can similarly define $D\{E_1,\dots, E_s, \Delta, E_{s+1},\dots, E_m\}$ for different anchor of $a_0$.
In particular, the BV-operator $\Delta$ (see (\ref{Delta_eq})) can be viewed as $\Delta(D)=D\{\Delta\}$, which also motivates our notation.

Recall that the standard brace relation specifies the expansion formula for the iterated brace operation
\[
D\{E_1,\dots, E_m\} \{F_1,\dots, F_n\}
\]
where braces are applied twice (see, e.g., \cite[Proposition 2.3.2]{tamarkin2005ring}).
In our situation, combining the above cyclic brace with the usual brace, we may encounter the following iterated brace operations
\begin{align*}
	&D\{E_1,\dots, E_m, \Delta \} \{F_1,\dots, F_n\} \\
	&D\{E_1,\dots, E_m\} \{F_1,\dots, F_n, \Delta\}
\end{align*}

It turns out that the expansion formula involves two distinct types of cyclic brace operations: the \textit{first-order cyclic brace} (Definition \ref{cyclic_brace_defn}), defined as above, and the \textit{second-order cyclic brace} (Definition \ref{cyclic_brace_diamond_defn}); see also the work of \cite{ward2012cyclic}.
We describe the latter for a simple case below: Given $D, E,F\in C^\bullet (A, A)$, we define $D\{E\{F, \lozenge \} \}$ in $C^\bullet(A, A)$ by
\[
\omega\big(D\{E\{F,\lozenge \}\}(a_1,\dots, a_k), a_0 \big) = \sum \pm \omega\big(
D(a_{r+1},\dots, E(\dots, F( \ ) ,\dots, a_0, \dots) , \dots, a_r)  \ , \ \one \big)
\]
Here, a similar but different symbol $\lozenge$ also anchors the position of $a_0$. However, the difference is that it simultaneously occupies positions within two nested braces. For example, we emphasize that in general, $D\{E\{F, \lozenge\}\} \neq D\{E\{F, \Delta\}\}$, where the latter is interpreted by first forming $G:=E\{F, \Delta\}$ (a first-order cyclic brace operation) and then applying the usual brace operation $D\{G\}$.
Higher-order cyclic braces could likely be defined similarly, but they are not needed for our purpose of establishing the BV structure on open-closed Hochschild cohomology (Theorem \ref{main_thm_intro}).
Now, we establish an explicit expansion formula for these operations, which we call the \textit{cyclic brace relations}:

\begin{thm}	[Cyclic Brace Relations]
	\label{cyclic_brace_relation_thm}
	Given Hochschild cochains $D,E_1,\dots, E_m,F_1,\dots, F_n $, we have the following two types of cyclic brace relations.
		\begin{align*}
		&\mathrm{(i)} \quad D\{E_1,\dots, E_m\} \{F_1,\dots, F_n,\Delta \} \\[1em]
		&=
		\sum_s \quad
		\sum_{\substack{i_1\leqslant j_1\leqslant \cdots \leqslant i_s\leqslant j_s}} \quad \   (-1)^{t} \ D \Big\{F_1,\dots, F_{i_1}, E_1\{F_{i_1+1}, \dots, F_{j_1}\}, F_{j_1+1},\dots, F_{i_2}, E_2\{F_{i_2+1},\dots, F_{j_2}\}, \dots, \\[-1em]
		& \qquad\qquad\qquad\qquad\qquad \qquad \qquad  , F_{i_s}, E_s\{ F_{i_s+1}, \dots, F_{j_s}\} , F_{j_s+1},\dots, F_n , \Delta, E_{s+1},\dots, E_m \Big \} \\[1.5em]
		&+
		\sum_s
		\sum_{\substack{i_1\leqslant j_1\leqslant \cdots \leqslant i_{s-1} \leqslant j_{s-1}\leqslant i_s}} (-1)^{t} \
		D \Big\{F_1,\dots, F_{i_1}, E_1\{F_{i_1+1}, \dots, F_{j_1}\}, F_{j_1+1},\dots, F_{i_2}, E_2\{F_{i_2+1},\dots, F_{j_2}\}, \dots, \\[-1em]
	& \qquad\qquad\qquad\qquad\qquad \qquad \qquad  , F_{i_{s}}, E_{s} \{ F_{i_{s}+1}, \dots, F_{n}, \lozenge \}   , E_{s+1},\dots, E_m \Big \}\\[-0.7em]
	\end{align*}
	where the signs are
	\[
	t= \sum_{k=1}^s |E_k| \sum_{a=1}^{i_k} |F_a| + \sum_{k=s+1}^m |E_k| \sum_{a=1}^{n}|F_a| 
	\]
	\vspace{0.2em}
		\begin{align*}
		& \mathrm{(ii)} \quad D\{E_1,\dots, E_m, \Delta\} \{F_1,\dots, F_n\} \\
		&=
		\sum_{r=0}^n
		\sum_{r+1\leqslant i_1\leqslant j_1\leqslant \cdots \leqslant i_m\leqslant j_m\leqslant n} (-1)^{\tau} \ D \Big\{F_{r+1},\dots, F_{i_1}, E_1\{F_{i_1+1}, \dots, F_{j_1}\}, F_{j_1+1},\dots, F_{i_2}, E_2\{F_{i_2+1},\dots, F_{j_2}\}, \dots, \\[-1em]
		& \qquad\qquad\qquad\qquad\qquad \qquad \qquad  , F_{i_m}, E_m\{ F_{i_m+1}, \dots, F_{j_m}\} , F_{j_m+1},\dots, F_n , \Delta, F_1,\dots, F_r \Big \} 
		%	\\[1.5em]
		%	&+
		%	\hspace{-1.5em}
		%	\sum_{\substack{r+1\leqslant i_1\leqslant j_1\leqslant \cdots \leqslant i_{m-1}\leqslant j_{m-1}\leqslant i_m\leqslant n\\ 0\leqslant s \leqslant r}}
		%	\hspace{-1em}
		%	(-1)^{\tau} \
		%	D \Big\{F_{r+1},\dots, F_{i_1}, E_1\{F_{i_1+1}, \dots, F_{j_1}\}, F_{j_1+1},\dots, F_{i_2}, E_2\{F_{i_2+1},\dots, F_{j_2}\}, \dots, \\[-1em]
	%	& \qquad\qquad\qquad\qquad\qquad \qquad \qquad  , F_{i_{m}}, E_{m} \{ F_{i_{m}+1}, \dots, F_{n}, \lozenge , F_1,\dots, F_s \}   , F_{s+1},\dots, F_r \Big \}\\[-0.7em]
	\end{align*}
	with the sign
	\[
	\tau=\sum_{k=1}^m|E_k| \sum_{a=r+1}^{i_k}|F_a| +  \sum_{p=1}^r \sum_{q=r+1}^n |F_p|  |F_q|
\]
\end{thm}

%	\begin{rmk}
%	Here the signs should refer to certain shifted degrees relevant to the application in the proof of Theorem~\ref{main_thm_intro} below.
%	More general cyclic brace relations can be formulated by allowing the symbol 
%	$\Delta$ to appear in positions other than the rightmost one. Namely, one may further study the expressions in the form
%	\[
%	D\{E_1,\dots, E_m\} \{F_1,\dots, F_n, \Delta, F_{n+1},\dots, F_{n+s}\} 
%	\]
%	\[ D\{E_1,\dots, E_m,\Delta, E_{m+1},\dots, E_{m+t}\} \{F_1,\dots, F_n\}
%	\]
%The process follows a similar approach but involves more intricate formulas than those in Theorem \ref{cyclic_brace_relation_thm}. Since establishing the BV algebra structure on the open-closed Hochschild cohomology, our main goal, does not require these extended relations, we omit their detailed treatment to maintain focus and clarity.
%\end{rmk}

Our claim is that the BV structure, whether on open-string (for $A_\infty$ algebras) or open-closed-string (for OCHAs) Hochschild cohomology, will be derived mostly from the cyclic brace relations.
Indeed, to show the cochain-level identity (\ref{[DE]_intro}) for an OCHA $\q$, we will apply it to the three circumstances: $\q\{D\}\{E,\Delta\}$, $D\{\q\}\{E,\Delta\}$, and $D\{E,\Delta\} \{\q\}$.
%For instance, the cyclic brace relation implies
%\begin{align*}
%	\q\{D\}\{E,\Delta\}&=\q\{D, E,\Delta\} + (-1)^{|D||E|} \q\{E, D,\Delta\} + (-1)^{|D||E|} \q\{E,\Delta, D\}  \\
%	&+ \q\{D\{E\}, \Delta\} + (-1)^{|D||E|} \q\{E, D\{\lozenge\}\}+ \q\{D\{E,\lozenge\}\}
%\end{align*}
%One may see that this is more or less similar to the usual brace relation for $\q\{D\}\{E, E'\}$ for another cochain $E'$, while the sign rule somewhat differs.
%The BV algebra structure in Theorem \ref{main_thm_intro} requires the cyclic symmetry condition for the given OCHA $(\ml,\q)$. This may suggest 
Besides, there is a connection between the cyclic property of $\q$ and cyclic braces. Indeed, for any cyclic Hochschild cochain $\q$, there exists a natural interchange between first-order and second-order cyclic braces, expressed roughly as follows (see Proposition \ref{cyclic_brace_with_omega_cyclic_prop} for more details): 
\[
\q \{ D\{E_1,\dots, E_m, \Delta, E_{m+1},\dots, E_n \} \} =
\pm \ D\{E_1,\dots, E_m, \q \{\lozenge \}, E_{m+1},\dots, E_n \} 
\]
\[
\q \{D_1,\dots, D_{s-1}, D_s\{\Delta\}, D_{s+1}, \dots, D_n\} = \pm 
D_s\{ 
\q \{D_{s+1},\dots, D_n, \lozenge, D_1,\dots, D_{s-1}\} \}
\]
For example, as \(\q\) is cyclic, we can show \(D\{\q\{\lozenge\}\} = \pm \q\{D\{\Delta\}\} = \pm \q\{\Delta D\}\) to obtain $ \delta \Delta = \Delta \delta$.

\subsection*{Acknowledgment}
The author is grateful to Jim Stasheff and Ben Ward for useful email correspondence, to Jiahao Hu, Ryszard Nest, Xinxing Tang, Yi Wang for useful conversations.

%\vspace{3em}

\section{Preliminaries}

Define $[k] := \{1, \dots, k\}$. By
$
I_1 \sqcup \cdots \sqcup I_r = [k]
$
we mean it gives a partition of $[k]$ into \textit{ordered} subsets $I_j = \{i_1 < i_2 < \cdots\}$ for $1 \leq j \leq r$.
Besides, we introduce the "dotted partition”
\begin{equation}
	\label{dot_sqcup_notation_eq}
	I_1 \dot\sqcup \cdots \dot\sqcup I_r = [k]
\end{equation}
which means the partition $I_1 \sqcup \cdots \sqcup I_r = [k]$ further satisfies that all elements in $I_i$ are smaller than those in $I_j$ if $i < j$. For example, $\{1,2\} \dot\sqcup \{3,4,5\}=[5]$ is a dotted partition, but $\{1,4\}\sqcup \{2,3,5\}=[5]$ is not.
Define $C^k(A,A')$ to be the space of multilinear maps $A^{\otimes k}\to A'$.
Define 
\begin{equation}
	\label{hochschild_cochain_eq}
	C^\bullet (A, A')=\prod_{k\geqslant 1} C^k(A,A')
\end{equation}
An \textit{\textbf{$A_\infty$ algebra}} is an element $\m=\{\m_k : k\geqslant 1\}$ in $C^\bullet (A, A)$ such that its degree is $|\m|=|\m_k|=1$, the first term $\m_{1}=d_A$ agrees with the differential, and
\[
\sum_{\substack{ k_1+k_2=k+1 \\ k_1, k_2\geqslant 1}}  \sum_{i=1}^{k_1+1} (-1)^\ast \ \m_{k_1} (a_1,\dots, \m_{k_2} (a_i,\dots, a_{i+k_2-1}), \dots, a_k ) =0
\]
where $\ast = \sum_{j=1}^{i-1} |a_j|$.
Note that the concise sign $|\m|=1$ is different from the standard convention but can be achieved by the use of shifted degree.

A multi-linear map $f: Z^{\otimes \ell}\to Z'$ of graded vector spaces $Z,Z'$ is said to be \textit{graded symmetric} if 
$
f(z_{\sigma(1)},\dots, z_{\sigma(\ell)})=(-1)^{\epsilon(\sigma)} f(z_1,\dots,z_\ell)
$
for $z_i\in Z$, all permutations $\sigma \in S_\ell$, and the sign
$
{\epsilon(\sigma)} = \epsilon(\sigma; z_1,\dots, z_\ell) = \sum_{i<j: \sigma(i)>\sigma(j)} |z_i| \cdot |z_j|
$.
Let $I_{\ell}$ be the subspace of $Z^{\otimes \ell}$ generated by 
$
z_1\otimes \cdots \otimes z_\ell - (-1)^{\epsilon(\sigma)} z_{\sigma(1)}\otimes \cdots \otimes z_{\sigma(\ell)}
$
for $z_i\in Z$, $\sigma \in S_\ell$.
Define 
$
Z^{\wedge \ell} := Z^{\otimes \ell} / I_{\ell}
$
and call it the $\ell$-th {graded symmetric tensor power} of $Z$.
Note that $Z^{\wedge \ell}$ can be characterized by a universal property: there is a canonical graded symmetric multilinear map $\varphi: Z^{\times \ell} \to Z^{\wedge \ell}$ such that for every graded symmetric multilinear map $f: Z^{\times \ell} \to E$, there is a unique linear map $f_{\wedge}: Z^{\wedge\ell}\to E$ with $f(z_1,\dots, z_\ell)=(-1)^{\epsilon(\sigma)} f_{\wedge}(\varphi(z_1,\dots, z_\ell))$.
From now on, we write $z_1\wedge\cdots \wedge z_\ell$ for the image $\varphi(z_1,\dots, z_\ell)$.
Then, $z_1\wedge \cdots \wedge z_\ell = (-1)^{\epsilon(\sigma)}z_{\sigma(1)}\wedge \cdots \wedge z_{\sigma(\ell)}$.

We introduce 
\[
z_J=z_{j_1}\wedge \cdots \wedge z_{j_n}
\]
for an ordered subset $J=\{j_1<\cdots <j_n\}$.
Abusing the notation, we also write 
\[
a_J=a_{j_1}\otimes \cdots\otimes a_{j_n}
\] if the context is clear.
Let $(Z,d_Z)$ and $(Z', d_{Z'})$ be differential graded vector spaces.
Define $\tilde C^\ell (Z, Z')$ to be the space of graded symmetric maps $Z^{\wedge \ell}\to Z'$.
Define the Chevally-Eilenberg-type space
\[
\tilde C^\bullet (Z, Z') =\prod_{\ell\geqslant 1} \tilde C^\bullet (Z, Z')
\]
An \textit{\textbf{$L_\infty$ algebra}} is an element $\ml=\{\ml_\ell: \ell\geqslant 1\}$ in $\tilde C^\bullet(Z, Z)$ such that $|\ml|=|\ml_{\ell}|=1$, $\ml_{1}=d_Z$, and 
\[
\sum_{J_1\sqcup J_2=[\ell]}  (-1)^\epsilon \ \ml_{|J_2|+1} ( \ml_{|J_1|} (z_{J_1}) \wedge z_{J_2} ) =0
\]
with $\epsilon$ given by $z_{[\ell]}=(-1)^{\epsilon} z_{J_1}\wedge z_{J_2}$.
For two $L_\infty$ algebras $(Z,\ml)$ and $(Z',\ml')$, an \textit{$L_\infty$ homomorphism} from $(Z,\ml)$ to $(Z',\ml')$ is an element $\mk=\{\mathfrak k_{\ell}: \ell\geqslant 1\}$ in $\tilde C^\bullet (Z, Z')$ such that $|\mk|=|\mk_{\ell}|=0$ and
\begin{align*}
	&\sum_{J_1\sqcup\cdots\sqcup J_r=[\ell]}  (-1)^{\epsilon_1} \  
	\ml'_{r}
	\big(
	\mk_{|J_1|} (z_{J_1}) \wedge \cdots \wedge \mk_{|J_r|} (z_{J_r})
	\big) =
	\sum_{K_1\sqcup K_2=[\ell] }
	(-1)^{\epsilon_2} \
	\mk_{|K_2|} \big( \ml_{|K_1|} (z_{K_1})\wedge z_{K_2} \big)
\end{align*}
with $\epsilon_1, \epsilon_2$ characterized by $z_{[\ell]}=(-1)^{\epsilon_1} z_{J_1}\wedge \cdots \wedge z_{J_r}$ and $z_{[\ell]} = (-1)^{\epsilon_2} z_{K_1}\wedge z_{K_2}$.
Remark that by definition, $\mk_{1}$ gives rise to a cochain map from $(Z, d_Z=\ml_{1})$ to $(Z', d_{Z'}=\ml'_{1})$.

Fix differential graded vector spaces $(A, d_A)$, $(A', d_{A'})$, and $(Z, d_Z)$. Define $C^{\ell, k} (Z ; A, A')$ to be the space of maps $\varphi: Z^{\wedge \ell}\otimes A^{\otimes k} \to A'$.
Define the open-closed Hochschild-type space as follows
\begin{equation}
	\label{hochschild_cochain_open_closed_eq}
	C^{\bullet,\bullet}(Z; A, A')= \prod_{\substack{\ell,k \geqslant 0 \\ (\ell,k)\neq (0,0)}} 
	C^{\ell, k} (Z ; A, A')
\end{equation}

Any element $\ml=\{\ml_\ell : \ell \geqslant 1\}$ in $\tilde C^\bullet (Z,Z)$ induces a map
\[
\widehat{\ml}: C^{\bullet,\bullet}(Z; A, A') \to C^{\bullet,\bullet}(Z; A, A')
\]
as follows:
For $D\in C^{\ell,k}(Z;A,A')$, $z_{[\ell]}=z_1\wedge \cdots \wedge z_\ell\in Z^{\wedge \ell}$, $a_{[k]}=a_1\otimes \cdots \otimes a_k \in A^{\otimes k}$, we define
\begin{equation}
	\label{hat_ml_eq}
\widehat{\ml}(D)_{\ell,k} (z_{[\ell]}; a_{[k]}) = \sum (-1)^\epsilon \  D_{|J_2|+1,k} (\ml_{|J_1|} (z_{J_1})\wedge z_{J_2}; a_{[k]})
\end{equation}
where the sum is taken over all the partitions $[\ell]=J_1\sqcup J_2$ and the sign $\epsilon$ is given by $z_{[\ell]}=(-1)^\epsilon z_{J_1}\wedge z_{J_2}$.
It is straightforward to verify that if $\ml$ is an $L_\infty$ algebra, then $\widehat{\ml} \ \ \widehat{\ml}=0$.
We call $\widehat{\ml}$ the \textbf{\textit{closed string action}} induced by $\ml$.

For $D, E_1,\dots, E_m\in C^{\bullet,\bullet}(Z; A, A')$, we define a new element $D\{E_1,\dots, E_m\}$ in $C^{\bullet,\bullet}(Z; A, A')$ by the following formula: for $z_{[\ell]}=z_1\wedge \cdots \wedge z_\ell$ and $a_{[k]}=a_1\otimes \cdots \otimes a_k$, we have
\begin{equation}
	\label{brace_oc_eq}
	\begin{aligned}
	\ \  D\{E_1,\dots, E_m\} (z_{[\ell]}; a_{[k]})  =
	\sum (-1)^\ast D \big( z_J; a_{I_0}, E_1 (z_{L_1}; a_{K_1}), a_{I_1},\dots, a_{I_{m-1}}, E_m(z_{L_m}; a_{K_m}) , a_{I_m}  \big) 
\end{aligned}
\end{equation}
where the summation is taken over the (dotted) partitions
$[\ell] = J\sqcup L_1 \sqcup \cdots \sqcup L_m$ and $[k]=I_0\dot\sqcup K_1 \dot \sqcup I_1 \dot\sqcup \cdots \dot\sqcup I_{m-1}\dot\sqcup K_m \dot\sqcup I_m$ (\ref{dot_sqcup_notation_eq}); where the sign is
\begin{align*}
	\ast 
	&=
	\sum_{j=1}^m \big(\sum_{i=0}^{j-1} |a_{I_i}|  +  \sum_{i=1}^{j-1} |a_{K_i}| \big) \big(|E_j|+|z_{L_j}| \big)  +
	\sum_{j=1}^m  \big( |z_J|+\sum_{i=1}^{j-1}|z_{L_i}|\big) |E_j| 
	\ \   +\epsilon
\end{align*}
with $\epsilon$ determined by
$
z_{[\ell]}= (-1)^\epsilon \ z_J\wedge z_{L_1}\wedge \cdots \wedge z_{L_m} 
$.
We call it the \textbf{\textit{(open-closed) brace operation}}.
The case $m=0$ is allowed, and then we set
$D\{\}=D$.
In our sign convention, we have
\begin{equation}
	\label{brace_sign_eq}
	|D\{E_1,\dots, E_m\}| = |D| + |E_1|+\cdots +|E_m|
\end{equation}

We can establish an almost identical brace relation for the open-closed braces as above. Moreover, the closed-string actions of elements in $\tilde{C}^\bullet(Z,Z)$ on $C^{\bullet,\bullet}(Z; A, A')$ interact with the open-closed brace operations $D\{E_1,\dots, E_m\}$ in the following way:

\begin{lem}
	\label{brace_brace_lem}
	We have the following equation
	\begin{align*}
		&D\{E_1,\dots, E_m\} \{F_1,\dots, F_n\} \\
		&=
		\sum_{i_1\leqslant j_1\leqslant \cdots \leqslant i_m\leqslant j_m} (-1)^\ast \ D \Big\{F_1,\dots, F_{i_1}, E_1\{F_{i_1+1}, \dots, F_{j_1}\}, F_{j_1+1},\dots, F_{i_2}, E_2\{F_{i_2+1},\dots, F_{j_2}\}, \dots, \\
		& \qquad \qquad \qquad  , F_{i_m}, E_m\{ F_{i_m+1}, \dots, F_{j_m}\} , F_{j_m+1},\dots, F_n \Big \}
	\end{align*}
	where the sign is
	$
	\ast = \sum_{k=1}^m |E_k| \sum_{a=1}^{i_k} |F_a|
	$.
Moreover, we have
	\begin{align*}
		& \widehat{\ml}\big(D\{E_1, \dots, E_m \} \big)  \\
		&
		=
		(-1)^{|\ml|\sum_{i=1}^m|E_i|} \ \widehat{\ml}(D) \{E_1,\dots, E_m\} 
		+
		\sum_{i=1}^m (-1)^{|\ml|\sum_{j=i+1}^m|E_j|}D\{E_1,\dots, E_{i-1}, \widehat{\ml}(E_i), E_{i+1},\dots, E_m\}
	\end{align*}
\end{lem}

\begin{proof}
	See \cite{Yuan_ocha_hochschild} for the proof.
\end{proof}

\section{Cyclic brace operations}

From now on, we fix a specific element \(\one \in A\) with \(|\one| = -1\). This will later serve as the unit of an open-closed homotopy algebra, but for now, it is simply a chosen element of $A$.
Let $\Bbbk$ be the ground coefficient ring. 

\begin{defn}
	\label{cyclic_defn}
	A non-degenerate
	bilinear map $\omega: A\otimes A\to \Bbbk$ is called a \textbf{constant symplectic structure} (cf. \cite[Definition 14]{ocha_kajiura_cmp}) if it satisfies
	\begin{itemize}
		\itemsep 0pt
		\item $\omega( a_1, a_2)= (-1)^{|a_1||a_2|+1} \omega (a_2, a_1)$
		\item there is an integer $|\omega|\in\mathbb Z$ such that $\omega(a_1,a_2)=0$ except for $|a_1|+|a_2|+|\omega|=0$.
	\end{itemize}
	Given a constant symplectic structure $\omega$, we call $D\in C^{\bullet,\bullet}(Z; A, A)$ \textbf{$\omega$-cyclic} (or just \textbf{cyclic}) if
	\[
	\omega (  D(z_1,\dots, z_\ell, a_1,\dots, a_k) , a_0 ) =(-1)^p \ \omega (D(z_1,\dots, z_\ell, a_0,a_1,\dots, a_{k-1}), a_k)
	\]
	where $p=|a_0| \sum_{i=1}^k|a_i|$.
	The subspace of $\omega$-cyclic elements is denoted by $C^{\bullet,\bullet}_\omega(Z; A, A)$.
\end{defn}

\subsection{BV operator and first-order cyclic brace}
The non-degeneracy of $\omega$ can be used to define the following operator (cf. e.g. \cite[p.2352]{tradler2008batalin})
\begin{equation}
	\label{Delta_eq}
	\Delta : C^{\bullet,\bullet+1}(Z; A, A) \to C^{\bullet,\bullet}(Z; A, A)
\end{equation}
by requiring $\Delta D$ satisfies
\[
\omega\Big( \Delta D(z_1,\dots, z_\ell; a_1,\dots, a_k), a_0\Big) = \sum_{i=0}^k (-1)^{|a_{[i]}|(|a_{[i+1,k]}|+|a_0|) } \ \omega \Big(  D(z_1,\dots, z_\ell; a_{i+1},\dots, a_k, a_0, a_1,\dots, a_i) , \one \Big)
\]
for any $z_1,\dots, z_\ell\in Z$ and $a_0,a_1,\dots, a_k\in A$.
By definition, the above expression is zero except for
\[
|\omega|+ |\Delta D| + |z_{[\ell]}|+ |a_{[0,k]}| = |\omega| + |D| + |z_{[\ell]}|+ |a_{[0,k]}|  +|\one|=0
\]
In particular, as $|\one|=-1$, we have
$
|\Delta D| =|D|-1
$.
%and
%\[
%\omega( a_0, \Delta D(z_1,\dots, z_\ell; a_1,\dots, a_k))
%=
%(-1)^{(|D|-1+|z_{[\ell]}|+|a_{[k]}|)|a_0|-1}
%\omega( \Delta D(z_1,\dots, z_\ell; a_1,\dots, a_k), a_0) 
%\]
We call the above map $D\mapsto \Delta D$ the \textit{BV operator}.

We introduce the following generalization, called \textit{cyclic braces}, of the BV operator $\Delta$ that combines the brace operations in (\ref{brace_oc_eq}).
Remark that the cyclic brace operations are not just defined on the subspace 
$C^{\bullet,\bullet}_\omega(Z; A, A)$ of $\omega$-cyclic elements, but can in fact be constructed directly on the full $C^{\bullet,\bullet}(Z; A, A)$ as long as we have a non-degenerate bilinear form $\omega$.

\begin{defn}
	[First-order Cyclic Brace]
	\label{cyclic_brace_defn}
	Given $D, E_1,\dots, E_m \in C^{\bullet,\bullet}(Z; A, A)$, we define a new element \[
	D \{ E_1,\dots, E_m, \Delta \}
	\]
	in $C^{\bullet,\bullet}(Z; A, A)$ through the non-degeneracy of $\omega$ as follows:
\begin{align*}
&\omega\Big(D \{ E_1,\dots, E_m,\Delta \} (z_{[\ell]}; a_{[k]}) , a_0 \Big) \\
&=
\sum (-1)^\ast
\omega
\Big( D(z_{L_0}; a_{r+1},\dots, a_{j_1}, E_1(z_{L_1}; a_{j_1+1}, \dots, a_{i_1}), \dots, E_m(z_{L_m}; a_{j_m+1},\dots, a_{i_m}), \dots, a_k, a_0, a_1,\dots, a_{r} ) \ , \ \one \Big)
\end{align*}
where the summation is taken over all the partitions $[\ell]=L_0\sqcup L_1\sqcup \cdots \sqcup L_m$ and
\[
0\leqslant r\leqslant j_1 \leqslant i_1 \leqslant j_2\leqslant i_2\leqslant \cdots\leqslant j_m\leqslant i_m \leqslant k
\]
and where the sign is given by
\[
\ast = \sum_{p=1}^m |a_{[r+1,j_p]}|\big(|E_p|+|z_{L_p}|\big) + \sum_{p=1}^m \sum_{q=0}^{p-1}|z_{L_q}||E_p| +\epsilon + |a_{[1,r]}| \big(|a_{[r+1,k]}|+|a_0|\big)
\]
with $\epsilon$ decided by
\[
z_{[\ell]}=(-1)^\epsilon \ z_{L_0} \wedge z_{L_1}\wedge \cdots \wedge z_{L_m}
\]
In a similar manner, one can also formulate the definition of 
\[
D\{E_1,\dots, E_s, \Delta, E_{s+1},\dots, E_n\}
\]
\end{defn}

\begin{rmk}
	Here we allow $m=0$ and observe $\Delta D= D\{\Delta\}$ for the BV-operator (\ref{Delta_eq}).
Intuitively, the symbol $\Delta$ anchors the position of $a_0$ in the formula. 
Moreover, by definition, one can check that
\[
|D\{E_1,\dots, E_n,\Delta\} = |D\{E_1,\dots, E_s, \Delta, E_{s+1},\dots, E_n\}|= |D| +|E_1|+\cdots+|E_n|-1
\]
\end{rmk}

\begin{rmk}
	\label{ambiguity_rmk}
	Now, we have two types of braces from (\ref{brace_oc_eq}) and Definition \ref{cyclic_brace_defn}. These braces allow us to construct various expressions, such as
\[
D\{E_1, E_2\{ F, \Delta\} , E_3\} \quad  , \quad  D\{ E_1, \Delta, E_2\{F\}\}
\]
We warn that the first expression is \textit{not} the same as the element $G\in C^{\bullet,\bullet}(Z; A, A)$ defined by
\[
\omega\Big( G(\dots ), a_0 \Big) = \sum \pm \omega \Big( D \big(\dots E_1 \dots E_2 (\dots F\dots a_0 \dots ) \dots E_3 \dots \big)
 \ , \ \one \Big)
\]
Specifically, we need to first take $E_2'=E_2\{F, \Delta\}$ as above and next use the usual brace (\ref{brace_oc_eq}) to define $D\{E_1,E_2', E_3\}$.
Nevertheless, it will be useful to introduce a notation for such expression, for which we are going to give the details below.
\end{rmk}

\subsection{Second-order cyclic braces}
To investigate a cyclic version of the brace relation, we want to expand the iterated brace operations $D\{E_1,\dots, E_m\} \{F_1,\dots, F_n,\Delta\}$ and $D\{E_1,\dots, E_m,\Delta \} \{F_1,\dots, F_n\}$.
However, the first-order cyclic brace (Definition \ref{cyclic_brace_defn}) proves insufficient for this purpose. Therefore, we introduce the following variant, which we refer to as the second-order cyclic brace for distinction.
 
\begin{defn} [Second-order Cyclic Brace]
\label{cyclic_brace_diamond_defn}
Given $D, E_1,\dots, E_m,F_1,\dots, F_n\in C^{\bullet, \bullet}(Z; A, A)$ and fixed $1\leqslant i\leqslant m$, $0\leqslant j\leqslant n$, we define
\begin{equation}
	\label{notation_double_brace}
	D\{E_1,\dots, E_i\{F_1,\dots, F_j, \lozenge  , F_{j+1},\dots, F_n\},\dots, E_m\}
\end{equation}
to be an element in $C^{\bullet,\bullet}(Z; A, A)$, denoted by $G$ for clarity, such that $\omega\Big(
G(z_{[\ell]}; a_{[k]}) \ , \ a_0 \Big)$ is equal to
\[
\sum \pm \omega \Big(
D \Big(\dots E_1() \dots E_{i-1}(), \dots, E_i \big(\dots F_1(),\dots, F_j()\dots a_0 \dots F_{j+1}()\dots F_n() \dots \big) \dots E_m() \dots \Big) \ , \ \one 
\Big)
\]
or more precisely
\begin{align*}
&
\sum (-1)^\dagger \
\omega\Big( D \big(z_P; a_{J_0}, E_1(z_{P_1}; a_{L_1}), a_{J_1},\dots, a_{J_{i-1}}, \\
&
\qquad\qquad  E_i\big(z_{Q}; a_{ I_{0}}, F_1(z_{ Q_{1}}; a_{ K_{1}}) , a_{ I_{1}}, \dots, a_{ I_{j-1}}, F_j(z_{ Q_j}; a_{ K_j}) , a_{ I''_{j}}, \\
& 
\qquad\qquad  a_0, a_{ I'_j}, F_{j+1}(z_{ Q_{j+1}}; a_{ K_{j+1}}), \dots , F_n(z_{Q_n}; a_{K_n}), a_{I_n} \big) \\
&
\qquad\qquad  a_{J_{i+1}}, \dots, a_{J_{m-1}}, E_m (z_{P_m}; a_{L_m}), a_{J_m}
\big) , \one
\Big)
\end{align*}
where the summation is taken over (cf. (\ref{dot_sqcup_notation_eq}))
\begin{align*}
[\ell] &=
P\sqcup P_1\sqcup\cdots \sqcup P_{i-1}\sqcup \Big( Q \sqcup Q_1\sqcup \cdots Q_n \Big) \sqcup P_{i+1}\sqcup \cdots \sqcup P_m   \\
[r] &= \Big( I_j'' \dot\sqcup K_{j+1} \dot\sqcup I_{j+1} \dot\sqcup \cdots \dot\sqcup K_n \dot\sqcup I_n \Big) \dot\sqcup J_{i+1} \dot\sqcup \cdots \dot\sqcup J_{m-1} \dot\sqcup L_m \dot\sqcup J_m
  \\
[r+1,k] &=
J_0\dot\sqcup L_1 \dot\sqcup J_1 \dot\sqcup \cdots  \dot\sqcup  L_{i-1}\dot\sqcup J_{i-1}  \dot\sqcup \Big(I_0\dot \sqcup K_1 \dot\sqcup I_1 \dot\sqcup \cdots \dot\sqcup K_j \dot\sqcup  I_j' \Big)
\end{align*}
for some $1\leqslant r\leqslant k$ and where the $\dagger$ is Koszul's sign. Specifically, if we denote the biggest number in $I_s$ (resp. $J_s$) by $i_s$ (resp. $j_s$), then this sign is
\begin{align*}
	\dagger &=
	\sum_{s=1}^m |a_{[r+1,j_s]}|(|E_s|+|z_{P_s}|) + \sum_{t=1}^n |a_{[r+1,i_t]}| (|F_t|+|z_{Q_t}|) \\ 
	&+ 
	\sum_{\mu=1}^i |E_\mu| (|z_P|+\sum_{\nu=1}^{\mu-1} |z_{P_\nu}|) + \sum_{\mu=i+1}^m |E_\mu| (|z_P|+\sum_{\nu=1}^{\mu-1}|z_{P_\nu}|+|z_Q|+\sum_{\lambda=1}^n |z_{Q_\lambda}|)  \\
	&+\sum_{u=1}^n |F_u| (|z_P|+\sum_{\nu=1}^{i-1}|z_{P_\nu}|+|z_Q|+\sum_{\lambda=1}^{u-1}|z_{Q_\lambda}|)
	\\
	&+
	\epsilon+ |a_{[r]}| (|a_{[r+1,k]}|+|a_0|)
\end{align*}
with $\epsilon$ determined by 
\[
z_{[\ell]}=(-1)^\epsilon \ z_P\wedge z_{P_1}\wedge \cdots z_{P_{i-1}}\wedge \Big(z_Q\wedge z_{Q_1}\wedge \cdots \wedge z_{Q_n} \Big) \wedge z_{P_{i+1}} \wedge \cdots \wedge z_{P_m}
\]
\end{defn}

Using the notation in Definition \ref{cyclic_brace_diamond_defn}, the discussion in Remark \ref{ambiguity_rmk} suggests that $D\{E_1, E_2\{F,\Delta\}, E_3\}$ differs from $D\{E_1,E_2\{F,\lozenge\}, E_3\}$.
This distinction also explains why we need two different types of notations.
In contrast, the second expression $D\{E_1,\Delta, E_2\{F\}\}$ in Remark \ref{ambiguity_rmk} meets no ambiguity.
The observation is that ambiguity arises when the symbol $\bullet=\Delta$ or $\lozenge$ is nested within at least two braces, as in expressions of the form $\dots \{ \dots \{\dots \bullet \dots\}\dots\}\dots$, but not arises when the symbol appears within a single brace such as an expression of the form $\dots \{\dots \bullet \dots \{\dots\} \dots\}\dots$.
This observation is also our motivation to name the first/second order cyclic braces in Definition \ref{cyclic_brace_defn} and \ref{cyclic_brace_diamond_defn}

\subsection{Cyclic brace relations}
\label{s_proof_cyclic_brace_relation}
Now, it is natural to ask if the usual brace relation in Lemma \ref{brace_brace_lem} has analogous properties in the setting of cyclic brace $D\{E_1,\dots, E_m,\Delta\}$ introduced in Definition \ref{cyclic_brace_defn}.
In this section, let's prove the cyclic brace relations in Theorem \ref{cyclic_brace_relation_thm} below:

\begin{proof}
	[Proof of Theorem \ref{cyclic_brace_relation_thm}]
Let's first address (i).
By definition, we have
\begin{align*}
& \omega \Big(
D\{E_1,\dots, E_m\} \{F_1,\dots, F_n, \Delta\} (z_{[\ell]}; a_{[k]}) \ , \ a_0 \Big) = 
\sum \pm \ \omega \Big( D\{E_1,\dots, E_m\} (\dots F_1 \dots F_n \dots a_0\dots ) \ , \ \one
\Big)
\end{align*}
Note that $a_0$ must be on the right side of all $F_i$'s.
Based on (\ref{brace_oc_eq}), further expanding the above expression may have two cases as follows.
The first case is that the inputs of all $E_i$'s do not involve $a_0$. Then, a term in this case has the expression
	\[
	D\big( \dots  F_{i_1} \dots E_1 (\dots F_{i_1+1} \dots F_{j_1} \dots) \dots E_s(\dots F_{i_s+1} \dots F_{j_s} \dots) \dots F_n \dots a_0 \dots E_{s+1} \dots E_m \dots \big)
	\]
for some $1\leqslant s\leqslant m$. Putting all these types of terms together yields the first sum in (i).
The second case is that one of $E_i$'s has input $a_0$, and such a term is given by
	\[
D\big( \dots  F_{i_1} \dots E_1 (\dots F_{i_1+1} \dots F_{j_1} \dots) \dots E_{s}(\dots F_{i_{s}+1} \dots F_n \dots a_0 \dots )  \dots E_{s+1} \dots E_m \dots \big)
	\]
These terms will form the second sum in (i).

Next, let's consider (ii). We first compute
\begin{align*}
	\omega\Big(
	&D\{E_1,\dots, E_m, \Delta\} \{F_1,\dots, F_n\} (z_{[\ell]}; a_{[k]}), a_0
	\Big) \\
	&=
	\sum \ \pm \omega\Big(
	D\{E_1,\dots, E_m, \Delta\} (a_1, \dots, F_1, \dots, F_n,\dots, a_n) \ , \ a_0
	\Big)
\end{align*}
For clarity, let's temporarily omit $F_i$'s. Then, when we expand it as Definition \ref{cyclic_brace_defn}, we obtain terms in the form $D\big( \dots E_1 \dots E_m \dots a_0\dots \big)$. In particular, $a_0$ cannot be an input of any $E_i$.
Concerning $F_i$'s, we note that the cyclic order of $F_1, \dots, F_n, a_0$ must be preserved.
If we assume $F_1,\dots, F_r$ are on the right hand side of $a_0$, then on the left hand side of $a_0$, the computations are almost the same as Lemma \ref{brace_brace_lem} while only $F_{r+1},\dots, F_n$ are involved.
These discussions justify (ii).
The signs are Koszul's signs; they are tedious but straightforward to check.
\end{proof}

\subsection{Cyclic brace with cyclic Hochschild cochains}
\label{s_cyclic_brac_with_cyclic_cochain}

Note that the BV structure in Theorem \ref{main_thm_intro} requires the cyclic symmetric condition for the prescribed OCHA.
This suggests the cyclic symmetric condition should be relevant to the cyclic brace operations. Specifically, let's prove:

\begin{prop}
	\label{cyclic_brace_with_omega_cyclic_prop}
	Suppose $F$ is an $\omega$-cyclic element in $C^{\bullet,\bullet}(Z; A, A)$ in the sense of Definition \ref{cyclic_defn}.
	Then, we can relate the first-order and second-order cyclic braces as follows:
	\[
	F \{ D\{E_1,\dots, E_m, \Delta, E_{m+1},\dots, E_n \} \} =
	(-1)^{e_1} \ D\{E_1,\dots, E_m, F \{\lozenge \}, E_{m+1},\dots, E_n \} 
	\]
	\[
	F\{D_1,\dots, D_{s-1}, D_s\{\Delta\}, D_{s+1}, \dots, D_n\} = (-1)^{e_2}\
	D_s\{ 
	F\{D_{s+1},\dots, D_n, \lozenge, D_1,\dots, D_{s-1}\} \}
	\]
	where the signs are
	\begin{align*}
		e_1&= 1+|F| \left(|D|+\sum_{p=1}^m|E_p|-1\right) \\
		e_2&=1+(|D_s|-1) \left(|F|+\sum_{a=1}^{s-1}|D_a|\right) +\sum_{a<s}|D_a|\sum_{a>s}|D_a|
	\end{align*}
\end{prop}

\begin{proof}
	%[Proof of Proposition \ref{cyclic_brace_with_omega_cyclic_prop}]
	Let's first prove the first relation.
For simplicity, we focus on the case 
$m=n$; the computation for the general case follows a quite similar process and is left to the interested reader.
Expanding $D\{E_1,\dots, E_m, F\{\lozenge\}\}$ as in Definition \ref{cyclic_brace_diamond_defn} yields the following expression
\[
\sum
(-1)^{\kappa_0} \ \omega\Big( D(z_{K_0}; a_{r+1}, \dots a_{j_1}, E_1(z_{K_1}; \ ) \dots, a_{j_m}, E_m(z_{K_m}; \ ) \dots F(z_J; a_{j+1}, \dots a_0,\dots, a_i) \dots, a_r) \ , \ \one \Big)
\]
where $[\ell]=K_0\sqcup K_1\sqcup\cdots\sqcup K_m\sqcup J$ and the sign is
\begin{align*}
	\kappa_0=\sum_{p=1}^m |a_{[r+1,j_p]}| (|E_p|+|z_{K_p}|) +
	\sum_{p=1}^m |E_p| \sum_{q=0}^{p-1}|z_{K_q}| + |F||a_{[r+1,j]}|+ |F| (|z_{K_0}|+\sum_{p=1}^m|z_{K_p}|)\\
	+
	\epsilon_0+ |a_{[1,r]}|(|a_{[r+1,k]}|+|a_0|)
\end{align*}
with $\epsilon_0$ given by $z_{[\ell]}=(-1)^{\epsilon_0} \ z_{K_0}\wedge z_{K_1}\wedge \cdots \wedge z_{K_m}\wedge z_J$.
Sometimes we may also denote it by
\[
\epsilon_0=:\epsilon(z_{[\ell]}\mapsto z_{K_0}\wedge z_{K_1}\wedge \cdots z_{K_m}\wedge z_J)
\]

For clarity, let's set $G=D\{E_1,\dots, E_m,\Delta\}$ temporarily. Using the fact that $F$ is $\omega$-cyclic and the skew-symmetry of $\omega$, we obtain
\begin{align*}
\omega \Big(
F\{G\} (z_{[\ell]}; a_{[k]}) \ , \ a_0 \Big) 
&=
\sum (-1)^{\kappa_1} \ \omega\Big(F(z_{J}; a_1,\dots, G(z_{K}; a_{i+1},\dots, a_j),\dots, a_k) \ , \ a_0 \Big) \\
&=
\sum (-1)^{\kappa_2} \ \omega\Big(
F(z_{J}; a_{j+1},\dots, a_k, a_0, a_1,\dots, a_i) \ , \ G(z_{K}; a_{i+1},\dots, a_j) \Big) \\
&=
\sum (-1)^{\kappa_3} \ \omega\Big(G(z_{K}; a_{i+1},\dots, a_j) \ , \ F(z_{J}; a_{j+1},\dots, a_k, a_0, a_1,\dots, a_i) \Big)
\end{align*}
where $J\sqcup K=[\ell]$ and the last sign is
\[
\kappa_3= |a_{[1,i]}|(|a_{[i+1,k]}|+|a_0|) + (|F|+|z_{J}|) (|G|+|z_{K}|+|a_{[i+1,j]}|) + |G||z_J|+ \epsilon_1+1
\]
with $\epsilon_1$ given by $z_{[\ell]}=(-1)^{\epsilon_1} z_{J}\wedge z_{K}$.
By Definition \ref{cyclic_brace_defn}, expanding $G=D\{E_1,\dots, E_m,\Delta\}$ in the above equation yields the same expression but with a potentially different sign
\[
\sum
(-1)^{\kappa_4} \ \omega\Big( D(z_{K_0}; a_{r+1}, \dots a_{j_1}, E_1(z_{K_1}; \ ) \dots, a_{j_m}, E_m(z_{K_m}; \ ) \dots F(z_J; a_{j+1}, \dots a_0,\dots, a_i) \dots, a_r) \ , \ \one \Big)
\]
where $K_0\sqcup K_1\sqcup\cdots \sqcup K_m=K$.
By a direct computation, one can find $\kappa_4=\kappa_0+1+ |F| |D\{E_1,\dots, E_m\}|$.
%\begin{align*}
%\kappa_4
%&=
%\kappa_3+ \sum_{p=1}^m |a_{[r+1,j_p]}| (|E_p|+|z_{K_p}|) + \sum_{p=1}^m \sum_{q=0}^{p-1} |z_{K_q}||E_p|
%+
%|a_{[i+1,r]}| \left( |a_{[r+1,k]}|+ |\q|+|z_J| + |a_0|+|a_{[1,i]}|\right) +\epsilon_2 \\
%&=\sum_{p=1}^m|a_{[r+1,j_p]}|(|E_p|+|z_{K_p}|)  + \sum_{p=1}^m |E_p| \sum_{q=0}^{p-1}|z_{K_q}|  
%+ |a_{[1,r]}|(|a_{[r+1,k]}|+|a_0|)
%\\
%& +(|D|-1+\sum_{p=1}^m|E_p|)(|\q|+|z_J|) +|\q| (|a_{[r+1,j]}| +|z_{K_0}|+\sum_{p=1}^m|z_{K_p}| )\\
%\end{align*}
%with $\epsilon_2$ given by $z_{K}=(-1)^{\epsilon_2} \ z_{K_0}\wedge z_{K_1}\wedge \cdots \wedge z_{K_m}$.

Next, we prove the second relation. 
We compute 
\begin{align*}
&\hspace{-2em} \omega\Big(F\{D_1,\dots, D_{s-1}, D_s\{\Delta\}, D_{s+1},\dots, D_n\} (z_{[\ell]}; a_{[k]}) \ , \ a_0 \Big) \\
&\hspace{-4.5em}=
\sum (-1)^{\xi_1} \
\omega\Big(F(z_J; a_1,\dots, D_1(z_{K_1}; a_{j_1+1},\dots, a_{i_1}) ,\dots,  D_s\{\Delta\}(z_{K_s}; a_{j_s+1},\dots, a_{i_s} ), \dots,  D_n(z_{K_n}; a_{j_n+1}, \dots, a_{i_n}) , \dots, a_k ) \ , \ a_0\Big) \\
&\hspace{-4.5em}=
\sum (-1)^{\xi_2} \
\omega\Big(F(z_J; a_{i_s+1},\dots, a_{j_{s+1}}, D_{s+1}(z_{K_{s+1}}; \ ) ,\dots, D_n(z_{K_n}; \ ), \dots, a_0,\dots, D_1(z_{K_1}; \ ), \dots, D_{s-1}(z_{K_{s-1}}; \ ), \dots) \ , \  D_s\{\Delta\} (z_{K_s}; \ )  \Big) \\
&\hspace{-4.5em}=
\sum (-1)^{\xi_3} \ \omega\Big(
D_s\{\Delta\} (z_{K_s}; a_{j_s+1}, \dots  ) \ , \
F(z_J;  \dots,  D_{s+1}(z_{K_{s+1}}; \ ) ,\dots, D_n(z_{K_n}; \ ), \dots, a_0,\dots, D_1(z_{K_1}; \ ), \dots, D_{s-1}(z_{K_{s-1}}; \ ), \dots, a_{j_s})
\Big) \\
&\hspace{-4.5em}=
\sum (-1)^{\xi_4} \ \omega\Big(
D_s (z_{K_s}; a_{r+1},\dots, F(z_J; a_{i_s+1}  \dots,  D_{s+1}( \ ) ,\dots, D_n( \ ), \dots, a_0,\dots, D_1(  \ ), \dots, D_{s-1}( \ ), \dots, a_{j_s}), \dots a_r) \ , \ \one
\Big) \\
&\hspace{-4.5em}=
\sum (-1)^{\xi_5} \ \omega\Big(
D_s\{F\{D_{s+1},\dots, D_n ,\lozenge, D_1,\dots, D_{s-1}\}\} (\dots) \ , \ a_0 \Big)
\end{align*}
Here the first equation just uses the usual open-closed brace in (\ref{brace_oc_eq}); the second equation follows from the cyclic property of $F$; the third equation uses the skew-symmetry of $\omega$; the fourth equation comes from the definition of $D_s\{\Delta\}$; the fifth equation is Definition \ref{cyclic_brace_diamond_defn}.
The primary challenge lies in computing the signs; we apologize for not listing all the details, but through a meticulous and direct calculation, we should be able to ultimately arrive at:
\begin{align*}
\xi_5&=
\sum_{a=1}^n (|D_a|+|z_{K_a}|) |a_{[1,j_a]}| - |a_{[1,j_s]}| + \sum_{a=1}^n|D_a| (|z_J|+\sum_{b=1}^{a-1}|z_{K_b}|) - |z_J|-\sum_{b=1}^{s-1}|z_{K_b}|+ \epsilon (z_{[\ell]}\mapsto z_J\wedge z_{K_1}\wedge \cdots\wedge z_{K_n}) \\
&+
\left(|a_{[1,i_s]}| +\sum_{a=1}^s(|D_a|+|z_{K_a}|) -1\right) \left( |a_{[i_s+1,k]}|+|a_0|+ \sum_{a={s+1}}^n (|D_a|+|z_{K_a}|)\right) \\
&
+ 1 +
\left( |D_s|-1+|z_{K_s}|+ |a_{[j_s+1,i_s]}| \right) \left( |F|+|z_J|+\sum_{a\neq s} (|D_a|+|z_{K_a}|) + |a_0|+|a_{[1,j_s]}|+|a_{[i_s+1, k]}| \right) \\
&+
\left( |a_{[r+1,k]}|+|a_0|+|a_{[1,j_s]}| + |F|+|z_J|+\sum_{a\neq s} (|D_a|+|z_{K_a}|)  \right)
|a_{[j_s+1,r]}| \\
&+
(|F|+|z_J|)|a_{[r+1,i_s]}|+|F||z_{K_s}| + \sum_{a=s+1}^n \big(|D_a|+|z_{K_a}|\big)|a_{[r+1,j_a]}|+\sum_{a=1}^{s-1}(|D_a|+|z_{K_a}|) (|a_{[r+1,k]}|+|a_0|+|a_{[1,j_a]}|)  \\
& \qquad + \sum_{a=s+1}^n |D_a| \left(|z_{K_s}|+|z_J|+\sum_{b=s+1}^{a-1}|z_{K_b}|\right) + \sum_{a=1}^{s-1} |D_a| \left(
|z_{K_s}|+|z_J|+\sum_{b=s+1}^n|z_{K_b}|+\sum_{b=1}^{a-1} |z_{K_b}| \right) \\
& \qquad + 
(|a_{[r+1,k]}|+|a_0|)|a_{[1,r]}| + \epsilon (z_{[\ell]}\mapsto z_{K_s}\wedge z_J\wedge z_{K_{s+1}}\wedge \cdots\wedge z_{K_n}\wedge z_{K_1}\wedge \cdots \wedge z_{K_{s-1}}) \\[1em]
&=
1+(|D_s|-1) \left(|F|+\sum_{a=1}^{s-1}|D_a|\right) +\sum_{a<s}|D_a|\sum_{a>s}|D_a|
\end{align*}
\end{proof}

\subsection{Closed string action and cyclic braces}
Finally, let's examine how the cyclic braces interplay with the closed string action in (\ref{hat_ml_eq}).
Recall that we have introduced the closed string action in (\ref{hat_ml_eq}) and have explored its relation to the usual (open-closed) brace in Lemma \ref{brace_brace_lem}.
Now, it is natural to further explore how it interacts with the cyclic braces as well.
Indeed, one key advantage of the cyclic brace approach to the BV structure in this paper is that it can be easily extended to the case of OCHA.

\begin{lem}
	\label{cyclic_brace_with_closed_string_action_lem}
	For any $\ml \in \tilde C^\bullet (Z, Z)$, one has
	\begin{align*}
		& \widehat{\ml}\big(D\{E_1, \dots, E_m, \Delta \} \big)  \\
		&
		=
		(-1)^{|\ml|\sum_{i=1}^m|E_i|} \ \widehat{\ml}(D) \{E_1,\dots, E_m, \Delta \} 
		+
		\sum_{i=1}^m (-1)^{|\ml|\sum_{j=i+1}^m|E_j|}D\{E_1,\dots, E_{i-1}, \widehat{\ml}(E_i), E_{i+1},\dots, E_m, \Delta \}
	\end{align*}
\end{lem}

The proof follows essentially the same argument as in Lemma \ref{brace_brace_lem}, so let's omit the details here.
We remark that when the symbol $\Delta$ appears in non-rightmost positions, the sign requires more careful verification. However, such cases are not needed for our main results, and we leave their detailed analysis to the interested reader.

\section{Normalized Hochschild cohomology}
An \textbf{\textit{OCHA}} (\textit{open-closed homotopy algebra}) is a tuple $(Z, A, \ml, \q)$, or simply a pair $(\ml,\q)$ if the context is clear, that consists of an $L_\infty$ algebra $(Z,\ml)$ with $\ml \in \tilde C^\bullet(Z, Z)$
and an element $\q=\{\q_{\ell,k}: \ell,k\geqslant 0, (\ell,k)\neq (0,0) \}$ in $C^{\bullet, \bullet}(Z; A, A)$ such that $|\q|=|\q_{\ell,k}|=1$ and
\begin{equation*}
	\q\{\q\} = \widehat{\ml} (\q)
	%		\sum
	%		(-1)^{\ast_1} \  \q_{|L_1|, |K_1|+|K_3|+1,\beta_1} (y_{L_1} ; x_{K_1} , \q_{|L_2|,|K_2| ,\beta_2} (y_{L_2} ; x_{K_2}), x_{K_3} )
	%		=
	%		\sum (-1)^{\ast_2} \  \q_{|J_2|+1, \beta_2} (\ml_{|J_1|,\beta_1} (y_{J_1})\wedge y_{J_2}; x_{[k]}) 
\end{equation*}
where use the braces in (\ref{brace_oc_eq}) and the closed-string action in (\ref{hat_ml_eq}).
%	A \textit{morphism of OCHAs} (or an \textit{OCHA homomorphism}) from $(Z,A,\ml,\q)$ to $(Z', A',\ml',\q')$ is a pair $\F=(\mk,\h)$ of an $L_\infty$ homomorphism $\mk=\{\mk_{\ell}\}: (Z,\ml) \to (Z', \ml')$ and an element $\h=\{\h_{\ell,k}\}$ in $C^{\bullet,\bullet} (Z; A, A')$ such that $|\h|=0$ and
%	\begin{align*}
	%		\q' \diamond_\mk \h  + \widehat{\ml} (\h) =\h\{\q\}
	%	\end{align*}
Here the degree may deviate from standard conventions, but we can achieve the desired sign by utilizing shifted degrees, if needed.

Following \cite{Yuan_ocha_hochschild}, there is an open-closed analogue of Hochschild cohomology defined as follows:
For $D, E\in C^{\bullet,\bullet}(Z; A, A)$, we put
$
[E,D]=E\{D\} - (-1)^{|D||E|} D\{E\}
$
for the brace in (\ref{brace_oc_eq}).
Given an OCHA $(\ml,\q)$, we define
the \textit{open-closed Hochschild differential}
\begin{equation}
	\label{delta_ocha_eq}
	\delta=\delta_{(\ml,\q)} : C^{\bullet,\bullet}(Z; A, A) \to C^{\bullet, \bullet}(Z; A, A)
\end{equation}
by
\begin{align*}
	\delta(D) =\q\{D\} - (-1)^{|D|} D\{\q\} + (-1)^{|D|} \widehat{\ml}(D)
	\equiv [\q, D] +(-1)^{|D|} \widehat{\ml}(D)
\end{align*}
Recall that $|\q|=|\ml|=1$ in our sign convention, so $|\q\{D\}|=|D\{\q\}|=|\widehat{\ml}(D)|=|D|+1$.
It is proved in \cite{Yuan_ocha_hochschild} that $\delta=\delta_{(\ml,\q)}$ is a differential, namely, $\delta^2=0$. Therefore, we can define the \textit{(open-closed) Hochschild cohomology} as the cohomology of this complex:
\[
HH(Z; A, A) = HH(Z; A, A)_{(\ml,\q)} :=H(C^{\bullet,\bullet}(Z;A,A), \delta)
\]

\begin{rmk}
An $A_\infty$ algebra $(A,\m=\{\m_k\})$ can be viewed as an OCHA $(\ml,\q)$ by setting $\ml_\ell=0$, $\q_{\ell,k}=0$ for $\ell>0$, and $\q_{0,k}=\m_k$.
One can check the open-closed Hochschild cohomology retrieves the ordinary Hochschild cohomology of the $A_\infty$ algebra $(A,\m)$.
\end{rmk}

\begin{defn}
	\label{unital_defn}
	(i) An OCHA $(Z, A, \ml, \q)$ is called \textbf{\textit{$\omega$-cyclic}} if $\q$ is $\omega$-cyclic in the sense of Definition \ref{cyclic_defn}.
	If the context is clear, we may just call it \textbf{\textit{cyclic}} for simplicity.
	(ii) An OCHA $(Z, A, \q, \ml)$ is called \textbf{unital} if there is an element $\one$, called a \textbf{unit}, with $|\one|=-1$ and
	\begin{itemize}
		\itemsep 0pt
		\item $\q_{0,2}(\one,a)=(-1)^{|a|-1}\q_{0,2}(a,\one)=a$;
		\item $\q_{\ell,k}(\dots; \dots \one \dots)=0$ for $(\ell,k)\neq (0,2)$ 
	\end{itemize}
\end{defn}

\subsection{Normalized Hochschild cochains}
\label{s_normalized_Hochshild}
From now on, we always assume $(Z,A,\ml,\q)$ is a cyclic unital OCHA with the unit $\one$.
An element $D$ in $C^{\bullet,\bullet}(Z; A, A)$ is called \textit{normalized} if $D(z_1,\dots, z_\ell; a_1,\dots, a_k)=0$ whenever one of $a_i$'s is $\one$ (cf. \cite[1.5.7]{loday2013cyclic}).

Denote by $\overline{C}^{\bullet,\bullet}(Z; A, A)$ the subspace of normalized Hochschild cochains.

\begin{prop}
	The normalized subspace $\overline{C}^{\bullet,\bullet}(Z; A, A)$ is preserved by the differential $\delta$; namely, if $D$ is normalized, then so is $\delta(D)$.
\end{prop}

\begin{proof}
Since \( \q \) is actually not normalized, additional care is needed.
Assume $D$ is a normalized cochain. 
For $0\neq i\neq k$, since $\q$ is unital and $D$ is normalized, we obtain
\begin{align*}
& [\q,D] (z_{[\ell]}; a_1,\dots, a_i, \one, a_{i+1},\dots, a_k)\\
&=
\sum(-1)^{1+|D|+|\q|(|a_{[i-1]}|+|z_{[\ell]}|)} \ D(z_{[\ell]}; a_1,\dots, \q_{0,2}(a_i,\one), a_{i+1},\dots, a_k)  \\
&+
\sum (-1)^{1+|D|+|\q|(|a_{[i]}|+|z_{[\ell]}|)} \ D(z_{[\ell]}; a_1,\dots, a_i, \q_{0,2}(\one, a_{i+1}), \dots, a_k) =0
\end{align*}
For $i=0$, we have
\begin{align*}
& [\q,D] (z_{[\ell]};  \one, a_1,\dots, a_k)\\
&=
(-1)^{|\one|(|D|+|z_{[\ell]}|)} \ \q_{0,2}(\one, D(z_{[\ell]}; a_1,\dots, a_k))
+(-1)^{1+|D|+|\q||z_{[\ell]}|} D(z_{[\ell]}; \q_{0,2}(\one, a_1), a_2,\dots, a_k) =0
\end{align*}
where we recall that $|\one|=-1$ and $|\q|=1$.
For $i=n$, the computation is similar and omitted.
Finally, since \( \widehat{\ml} \) does not involve inputs from $A$ and the normalization condition depends only on the unit \( \one \) in $A$, it follows that \( \widehat{\ml}(D) \) is also normalized.
\end{proof}

Thanks to the above proposition, we can define the \textbf{\textit{normalized (open-closed) Hochschild cohomology}} as
\[
\overline{HH}(Z;A,A) = \overline{HH}(Z;A,A)_{(\ml,\q)} = H(\overline{C}^{\bullet,\bullet}(Z; A, A), \delta)
\]
By construction, the first-order and second-order cyclic braces in Definitions \ref{cyclic_brace_defn} and \ref{cyclic_brace_diamond_defn} preserve the normalized subspace $\overline{C}^{\bullet,\bullet}(Z; A, A)$.

\subsection{Cyclic braces with cyclic and unital OCHA}

The following lemma describes additional properties of the first-order cyclic brace $\q\{D_1,\dots, D_s,\Delta, D_{s+1},\dots, D_n\}$
when $\q$ belongs to a cyclic and unital OCHA. The behavior differs significantly in three distinct cases: $n=0, n=1, n\geqslant 2$

\begin{lem}
	\label{q_first_order_brace_lem}
	\begin{enumerate}
		\item $
		\q\{\Delta\}=\Delta\q=0
		$
		\item $
		\q\{D, \Delta\}=-\q\{\Delta, D\} =D
		$
		\item For $0\leqslant s\leqslant n$ with $n\geqslant 2$, we have
		\[
		\q\{D_1,\dots, D_s, \Delta, D_{s+1},\dots, D_n\} =0
		\]
	\end{enumerate}
	
\end{lem}

\begin{proof}
	(1) Recall $\q\{\Delta\}=\Delta\q$ by definition. Since $\q$ is cyclic and $\one$ is a unit, we conclude that
	\[
	\omega\Big( (\Delta \q)_{\ell,k}  (z_{[\ell]}; a_1,\dots, a_k) , a_0 \Big)
	=
	\sum  (-1)^{|a_{[i]}|(|a_{[i+1,k]}|+|a_0|) } \ \omega(  \q(z_{[\ell]}; a_{i+1},\dots, a_k, a_0, a_1,\dots, a_i) , \one ) =0
	\]
	whenever $\ell\geqslant 1$ or $k\geqslant 2$. 
	It remains to discuss the exceptional cases.
	Recall that $(\ell,k)\neq (0,0)$ by definition. Thus, it suffices to assume $\ell=0$ and $k=1$,
	\begin{align*}
		\omega \Big( \Delta\q(a_1), a_0\Big) 
		&= \omega \Big(\q_{0,2}(a_1,a_0),\one \Big) +  (-1)^{|a_0||a_1|}\omega \Big( \q_{0,2}(a_0,a_1), \one \Big) \\
		&= (-1)^{|a_0|+|a_1|} \Big( \omega \Big(\q_{0,2}(\one, a_1),a_0 \Big) +  (-1)^{|a_0||a_1|}\omega \Big( \q_{0,2}(\one, a_0), a_1 \Big) \Big) \\
		&= (-1)^{|a_0|+|a_1|} \Big( \omega (a_1,a_0) + (-1)^{|a_0||a_1|}\omega (a_0, a_1)
		\Big)
	\end{align*}
	This vanishes exactly due to the skew-symmetry of $\omega$. Hence, $\q\{\Delta\}=\Delta\q=0$.
	
(2)	We first observe that $\q\{\Delta\}\{D\}=\q\{D,\Delta\}+\q\{\Delta, D\}=0$ by Theorem \ref{cyclic_brace_relation_thm} and the item (1) proved above. Thus, it suffices to compute $\q\{D,\Delta\}$. We compute
\begin{align*}
	\omega \Big(
	\q\{D,\Delta\} (z_{[\ell]}; a_{[k]}) \ , \ a_0 \Big)
	&=\sum(-1)^{\ast} \omega\Big(\q(z_J; a_{r+1},\dots, D(z_K; a_{j+1},\dots, a_i), \dots, a_k, a_0,\dots, a_r) \ , \ \one \Big) \\
	&= \omega\Big(\q_{0,2}(D({z_{[\ell]}}; a_{[k]}) , a_0 ) \ , \ \one \Big)
	= (-1)^{|D|+|z_{[\ell]}|+|a_{[k]}|+|a_0|} \ \omega\Big(\q_{0,2} (\one, D(z_{[\ell]}; a_{[k]}) ) \ , \ a_0 \Big) \\
	&=
	\omega\Big(D(z_{[\ell]}; a_{[k]}) \ , \ a_0\Big)
\end{align*}

(3) Finally, the last item of the lemma easily follows from the cyclic and unital conditions.
\end{proof}

Similarly, for $\q$ in a cyclic unital OCHA, the second-order cyclic braces involving $\q$ admit extra properties. Rather than developing the most general theory, we focus specifically on the cases needed for our applications:

\begin{lem}
	\label{q_second_order_brace_lem}
	We have
	\begin{align*}
		\q\{D\{\lozenge\}\}&= \q\{D\{E,\lozenge\}\}=0 \\
		\q\{D\{\lozenge\}, E\} &= (-1)^{1+|D||E|} \ \q\{E, D\{\lozenge\}\}
	\end{align*}
\end{lem}

\begin{proof}
	The cyclic and unital properties imply
	\begin{align*}
		&\omega\Big(\q\{D\{\lozenge\}\}(z_{[\ell]}; a_{[k]}) \ , a_0 \Big)
		=
		\sum (-1)^{\delta_1} \
		\omega\Big( \q(z_J; a_{r+1},\dots, a_j, D(z_K; a_{j+1},\dots, a_k , a_0, a_1,\dots, a_i ) , \dots, a_r ) \ , \one \Big) \\
		&=
		\sum (-1)^{|a_{[1,r]}|(|a_{[r+1,k]}|+|a_0|)} \ \omega\Big( \q_{0,2}(D(z_{[\ell]}; a_{r+1},\dots, a_k, a_0, a_1,\dots, a_{r-1}) , a_r) \ , \ \one \Big) \\
		&+
		\sum (-1)^{(|D|+|z_{[\ell]}|)|a_r| + |a_{[1,r-1]}|(|a_{[r,k]}|+|a_0|)} \ \omega \Big( \q_{0,2} (a_{r} , D(z_{[\ell]}; a_{r+1},\dots, a_k, a_0, a_1,\dots, a_{r-1}) ) , \one  \Big) \\
		&=
		\sum (-1)^{|a_{[1,r]}|(|a_{[r+1,k]}|+|a_0|)+|D|+|z_{[\ell]}|+ |a_{[0,k]}|} \ \omega\Big( \q_{0,2}(\one , D(z_{[\ell]}; a_{r+1},\dots, a_k, a_0, a_1,\dots, a_{r-1}))  , a_r  \Big) \\
		&+
		\sum (-1)^{(|D|+|z_{[\ell]}|)|a_r| + |a_{[1,r-1]}|(|a_{[r,k]}|+|a_0|) +|D|+|z_{[\ell]}|+ |a_{[0,k]}|} \ \omega \Big( \q_{0,2} (\one, a_{r}) \ , \ D(z_{[\ell]}; a_{r+1},\dots, a_k, a_0, a_1,\dots, a_{r-1} )\Big) \\
		&=
		\sum (-1)^{|a_{[1,r-1]}|(|a_{[r,k]}|+|a_0|)+|D|+|z_{[\ell]}|+ |a_{[0,k]}|}
		\Big( \omega (D \ , \ a_r ) + (-1)^{(|D|+|z_{[\ell]}|)|a_r| + (|a_0|+|a_{[1,r-1]}|+|a_{[r+1,k]}|) |a_r|} \  \omega (a_r \ , \ D ) 
		\Big)=0
	\end{align*}
	where the first equation comes from the unitality, the second equation is derived from the cyclicity, and the last equation follows from the unitality and the skew-symmetry of $\omega$.
	Besides, the proof of $\q\{D\{E,\lozenge\}\}=0$ is almost the same.
	On the other hand, we compute
	\begin{align*}
		& \omega\Big( \q\{D\{\lozenge\}, E\} (z_{[\ell]}; a_{[k]}) \ , \ a_0 \Big)  \\
		&=
		\sum (-1)^{\epsilon_1} \omega \Big(\q(z_{J_0}; a_{r+1},\dots, D(z_{J_1}; a_{s+1},\dots, a_0, \dots, a_t), \dots, E(z_{J_2}; a_{j+1},\dots, a_i), \dots, a_r) \ , \ \one \Big) \\
		&=
		\sum (-1)^{\epsilon_2} \ \omega\Big(\q_{0,2}(D(z_{K}; a_{s+1},\dots, a_k, a_0,\dots, a_t), E(z_{L}; a_{t+1},\dots, a_s) ) \ , \ \one \Big) \\
		&=
		\sum (-1)^{\epsilon_3} \omega\Big(\q_{0,2} (\one, D(z_{K}; a_{s+1},\dots, a_0,\dots, a_t) )  \ , \ E(z_L; a_{t+1},\dots, a_s)\Big) \\
		&=
		\sum (-1)^{\epsilon_4} \omega\Big(E(z_L; a_{t+1},\dots, a_s) \ , \ \q_{0,2}(\one, D(z_K; a_{s+1},\dots, a_0,\dots, a_t))\Big) \\
		&=
		\sum (-1)^{\epsilon_5} \omega\Big(
		\q_{0,2}(\one,E(z_L; a_{t+1},\dots, a_s)) \ , \ D(z_K; a_{s+1},\dots, a_0,\dots, a_t)\Big) \\
		&=
		\sum (-1)^{\epsilon_6} \omega\Big(\q_{0,2} (E(z_L; a_{t+1},\dots, a_s) , D(z_K; a_{s+1},\dots, a_0,\dots, a_t)) \ , \ \one \Big) =
		(-1)^{1+|D||E|} \omega ( \q\{E, D\{\lozenge\}\} , a_0)
	\end{align*}
	where $J_0\sqcup J_1\sqcup J_2=K\sqcup L=[\ell]$ and where the second equation holds since $\one$ is a unit.
\end{proof}

\subsection{BV operator}
Since the cyclic braces preserve the normalized subspace, the BV operator $\Delta$ from (\ref{Delta_eq}) restricts to an operator on $\overline{C}^{\bullet,\bullet}(Z; A, A)$. Let's still denote this induced operator on the normalized subspace by $\Delta$.

\begin{prop}
	\label{Delta_square_zero_prop}
	$\Delta^2=0$ and $\delta \circ \Delta= \Delta \circ \delta $. In particular, it induces an operator, still denoted by $\Delta$, on the normalized Hochschild cohomology $\overline{HH}(Z; A, A)$ (Section \ref{s_normalized_Hochshild}).
\end{prop}

\begin{proof}
	The first result is straightforward to check. For the second, using Lemma \ref{cyclic_brace_with_closed_string_action_lem} implies that the closed string action $\widehat{\ml}$ commutes with $\Delta$, namely, $\widehat{\ml} (\Delta D) = \Delta (\widehat{\ml} (D))$. Accordingly, it remains to check $[\q,\Delta D] = \Delta ([\q, D])$.
	In fact, by the cyclic brace relations in Theorem \ref{cyclic_brace_relation_thm}, we have
	\begin{align*}
		(\Delta D)\{\q\} \equiv D\{\Delta\} \{\q\} &=D\{\Delta, \q\} +D\{\q,\Delta\} \\
		\Delta(D\{\q\})\equiv D\{\q\}\{\Delta\} &= D\{\Delta, \q\} + D\{\q,\Delta\} + D\{\q\{\lozenge\}\}
	\end{align*}
	Since $\q$ is $\omega$-cyclic, it follows from Proposition \ref{cyclic_brace_with_omega_cyclic_prop} that 
	\[
	D\{\q\{\lozenge\}\}= (-1)^{1+(|D|-1)|\q|} \q\{D\{\Delta\}\} =(-1)^{1+(|D|-1)|\q|} \q\{\Delta D\}
	\]
	Therefore,
	\[
	\Delta (D\{\q\}) = (\Delta D)\{\q\}  - (-1)^{(|D|-1)|\q|} \q\{\Delta D\} = [\Delta D, \q]
	\]
	Then, further using Theorem \ref{cyclic_brace_relation_thm} (ii) yields
	\[
	\Delta(\q\{D\}) = \q\{D\} \{\Delta\} = \q\{D,\Delta\} +\q\{\Delta, D\} + \q\{D\{\lozenge\}\} = \q\{D\{\lozenge\}\}=0
	\]
	where the vanishing is because of Lemma \ref{q_second_order_brace_lem}.
\end{proof}

\subsection{Cochain-level identities for the BV relation}
As noted before in (\ref{[DE]_intro}), a primary objective of this work is to establish a Hochschild cochain identity that governs the BV structure on Hochschild cohomology. The following result achieves this goal.

\begin{thm}
	\label{[DE]_thm}
	We have
	\begin{align*}
		[D, E] &= \Delta (\q\{D, E\} ) + \q \{ \Delta D, E\} - (-1)^{|D||E|} \q\{\Delta E, D\}  \\
		&+
		\delta \left( D\{E,\Delta\} \right)- (-1)^{|D||E|} \delta\left( E\{D, \Delta\}  \right) \\
		&+
		(\delta D) \{E, \Delta\} - (-1)^{|D||E|} (\delta E) \{D, \Delta\} \\
		&+
		(-1)^{|D|} \left( D\{\delta E, \Delta\} - (-1)^{|D||E|} E\{\delta D, \Delta\} \right)
	\end{align*}
Therefore, in the normalized Hochschild cohomology
$\overline{HH}(Z; A,A)$, we have (cf. (\ref{BV_eq_intro})):
\[
[D, E]=\Delta (D\smallsmile E)+ \Delta D\smallsmile E - (-1)^{|D||E|} \ \Delta E\smallsmile D  
\]
\end{thm}

The above cochain-level identities above facilitate the proof of the BV structure for open-closed Hochschild cohomology.
Meanwhile, this result naturally extends to both associative and $A_\infty$ algebras, as they can be regarded as special cases of OCHAs.

To prove Theorem \ref{[DE]_thm}, we need the following two lemmas. One one hand, we build on the previous results, using the cyclic and unital properties of the OCHA, to appropriately gather $\delta$-boundary elements as follows.
\begin{lem}
	\label{D{E}_lem}
	One has
	\[
	D\{E\} = \q\{\Delta D, E\} +\q\{D\{\lozenge\}, E\} + \delta\big(D\{E,\Delta\}\big) + (\delta D)\{E,\Delta\} + (-1)^{|D|} \ D\{\delta E,\Delta\}
	\]
\end{lem}

\begin{proof} 
	Since $\q$ is cyclic, it follows from Proposition \ref{cyclic_brace_with_omega_cyclic_prop} that
	\begin{equation*}
		\label{proof_BV_1}
		\q\{D\{E,\Delta\} \} = (-1)^{|D|+|E|} D\{E, \q\{\lozenge\}\}
	\end{equation*}
	By the type (ii) of the cyclic brace relations (Theorem \ref{cyclic_brace_relation_thm}), we obtain
	\begin{align*}
		\label{proof_BV_2}
		D\{E,\Delta\} \{\q\} = D\{E, \Delta, \q\} + D\{E,\q,\Delta\} + (-1)^{|E|} D\{ \q, E,\Delta\} + D\{ E\{\q\} , \Delta\}
	\end{align*}
	Due to the type (i) of the cyclic brace relations (Theorem \ref{cyclic_brace_relation_thm}), we also get
	\begin{align*}
		\q\{D\}\{E,\Delta\} 
		&=  \q\{D, E, \Delta\}    +(-1)^{|D||E|}  \q\{E, D, \Delta\}  
		+(-1)^{|D||E|}  \q\{E, \Delta, D\}  \\
		&
		+ \ \q\{D\{E\}, \Delta\} 
		+ (-1)^{|D||E|} \q\{E, D\{\lozenge\}\} 
		+ \q\{D\{E,\lozenge\}\}
		\\[1em]
		D\{\q\}\{E,\Delta\}
		&
		=D\{\q, E, \Delta\} + (-1)^{|E|} D\{E,\q,\Delta\} + (-1)^{|E|} D\{E,\Delta, \q\}  \\
		&+ D\{\q\{E\}, \Delta\} + (-1)^{|E|} D\{E, \q\{\lozenge\}\} + D\{\q\{E,\lozenge\}\}
	\end{align*}
By Lemma \ref{cyclic_brace_with_closed_string_action_lem}, we have
$
		\widehat{\ml} (D\{E,\Delta\}) =
		(-1)^{|E|} \widehat{\ml }(D) \{E, \Delta\} + D\{\widehat{\ml}(E), \Delta\} 
$
	for the corresponding closed string actions.
	Considering the definition of $\delta$ in (\ref{delta_ocha_eq}) and recalling $|D\{E,\Delta\}|=|D|+|E|-1$,
	the above five equations with a direct computation implies that
	\begin{align*}
		&\hspace{-1em} \delta\left( D\{E,\Delta\} \right) + (\delta D)\{E,\Delta\}  + (-1)^{|D|} \ D\{ \delta E, \Delta\}  \\
		&=
		\q\{D\{E\}, \Delta\} +\q\{D, E,\Delta\} + (-1)^{|D||E|} \q\{E, D,\Delta\} +(-1)^{|D||E|} \q\{E,\Delta, D\} \\
		& \qquad\qquad\qquad\qquad \qquad\qquad\qquad\qquad +(-1)^{|D||E|} \q\{E, D\{\lozenge\}\} 
		+ \q\{D\{E,\lozenge\}\}
		+ (-1)^{|D|-1} D\{\q\{E,\lozenge\}\} 
	\end{align*}
	By Lemma \ref{q_first_order_brace_lem}, we have \[
	\q\{D\{E\}, \Delta\}=D\{E\}
	\]
	and \[
	\q\{D,E,\Delta\}=\q\{E,D,\Delta\}=\q\{E,\Delta,D\}=0
	\]
	By Proposition \ref{cyclic_brace_with_omega_cyclic_prop}, we obtain \[
	\q\{\Delta D, E\} = \q\{D\{\Delta\}, E\}=(-1)^{|D|} \ D\{\q \{E,\lozenge\}\}
	\]
	By Lemma \ref{q_second_order_brace_lem}, we know \[
	\q\{D\{E,\lozenge\}\}=0
	\] and 
	\[
	\q\{D\{\lozenge\}, E\} = (-1)^{1+|D||E|} \ \q\{E, D\{\lozenge\}\}
	\]
	Putting them together, the above expression equals to
$D\{E\} -\q\{\Delta D, E\} - \q\{D\{\lozenge\}, E\}$.
\end{proof}

On the other hand, we use the cyclic braces to study $\Delta (D\smallsmile E)$ at the Hochschild cochain level as follows:

\begin{lem}
	\label{D{E}_2_lem}
	One has
	\[
	\Delta(\q\{D,E\}) =
	\q\{D\{\lozenge\}, E\} -(-1)^{|D||E|} \ \q\{ E\{\lozenge\}, D \}
	\]
\end{lem}

\begin{proof}
	By Theorem \ref{cyclic_brace_relation_thm} and Lemma \ref{q_first_order_brace_lem}, we obtain
	\begin{align*}
		\Delta(\q\{D,E\})
		&=
		\q\{D,E\}\{\Delta\} = \q\{D, E,\Delta\}+\q\{D,\Delta, E\} +\q\{\Delta, D, E\} +\q\{D\{\lozenge\}, E\} +\q\{D,E\{\lozenge\}\} \\
		&=
		\q\{D\{\lozenge\}, E\} +\q\{D, E\{\lozenge\}\}
	\end{align*}
	Further using Lemma \ref{q_second_order_brace_lem} yields the result.
\end{proof}

\begin{proof}[Proof of Theorem \ref{[DE]_thm}]
	Use Lemma \ref{D{E}_lem} and \ref{D{E}_2_lem}.
\end{proof}

\subsection{Proof of Theorem \ref{main_thm_intro}}
First, we observe that both the braces and the cyclic braces preserve the normalized conditions, the brace relations in Lemma \ref{brace_brace_lem} still hold within $\overline{C}^{\bullet,\bullet}(Z; A, A)$. Therefore, using almost the same argument as in
\cite{Yuan_ocha_hochschild}, we can establish the Gerstenhaber algebra structure on the normalized open-closed Hochschild cohomology (Section \ref{s_normalized_Hochshild}), where the cup product and the Gerstenhaber bracket are also defined at the cochain level by $D\smallsmile E=\q\{D, E\}$ and $[D, E]=D\{E\} - (-1)^{|D||E|} E\{D\}$.
The BV algebra structure is established by Proposition \ref{Delta_square_zero_prop} and Theorem \ref{[DE]_thm}.

\bibliographystyle{alpha}
\bibliography{mybib_bv}	

\end{document}